\documentclass[12pt]{article}

\usepackage{ifxetex,ifluatex}
\newif\ifxetexorluatex
\ifxetex
  \xetexorluatextrue
\else
  \ifluatex
    \xetexorluatextrue
  \else
    \xetexorluatexfalse
  \fi
\fi

\ifxetexorluatex
  \usepackage{fontspec}
\else
  \usepackage[T1]{fontenc}
  \usepackage[utf8]{inputenc}
  \usepackage{lmodern}
\fi

\usepackage{hyperref}

\usepackage{amsmath}
\usepackage{amsfonts}
\usepackage{mathrsfs}
\usepackage{float}
\restylefloat{table}
\usepackage[backend=bibtex,
style=authoryear,
sorting=nty
]{biblatex}
\usepackage[titletoc,title]{appendix}
\usepackage{chngcntr}
\usepackage{apptools}
\usepackage{amsthm}
\usepackage[titletoc]{appendix}
\usepackage{cleveref}
\usepackage{graphicx}
\graphicspath{ {./} }
\usepackage{geometry}
\newtheorem{theorem}{Theorem}

\newtheorem{proposition}{Proposition}
\newcommand{\R}{\mathbb{R}}
\newcommand{\C}{\mathbb{C}}

\newcommand{\cX}{\mathcal{X}}
\newcommand{\cC}{\mathcal{C}}
\newcommand{\cG}{\mathcal{G}}
\newcommand{\cH}{\mathcal{H}}
\newcommand{\cGO}{\mathcal{GO}}
\newcommand{\cF}{\mathcal{F}}
\newcommand{\bpi}{\boldsymbol{\pi}}
\newcommand{\cK}{\mathcal{K}}
\newcommand{\cQ}{\mathcal{Q}}
\newcommand{\cR}{\mathcal{R}}
\newcommand{\cS}{\mathcal{S}}
\newcommand{\fll}{\mathfrak{l}}
\newcommand{\bL}{\boldsymbol{L}}
\newcommand{\bW}{\boldsymbol{W}}
\newcommand{\bF}{\boldsymbol{F}}
\newcommand{\bQ}{\boldsymbol{Q}}
\newcommand{\bG}{\boldsymbol{G}}
\newcommand{\bS}{\boldsymbol{S}}
\newcommand{\bU}{\boldsymbol{U}}
\newcommand{\bx}{\boldsymbol{x}}
\newcommand{\by}{\boldsymbol{y}}

\newcommand{\bv}{\boldsymbol{v}}

\newcommand{\bX}{\boldsymbol{X}}
\newcommand{\bI}{\boldsymbol{I}}
\newcommand{\bA}{\boldsymbol{A}}
\newcommand{\bB}{\boldsymbol{B}}
\newcommand{\bC}{\boldsymbol{C}}
\newcommand{\bO}{\boldsymbol{O}}
\newcommand{\bD}{\boldsymbol{D}}
\newcommand{\bJ}{\boldsymbol{J}}
\newcommand{\bK}{\boldsymbol{K}}

\newcommand{\Xlag}{\boldsymbol{X}_{\textsc{LAG}}}

\newcommand{\bY}{\boldsymbol{Y}}
\newcommand{\bH}{\boldsymbol{H}}

\newcommand{\bT}{\boldsymbol{T}}

\newcommand{\bP}{\boldsymbol{P}}
\newcommand{\Gperp}{\boldsymbol{G}^o_{\perp}}

\newcommand{\bZ}{\boldsymbol{Z}}

\newcommand{\bep}{\boldsymbol{\epsilon}}

\newcommand{\bPhi}{\boldsymbol{\Phi}}

\newcommand{\baP}{\boldsymbol{\bar{P}}}
\newcommand{\baq}{\boldsymbol{\bar{q}}}

\newcommand{\bXLag}{\boldsymbol{X}_{\textsc{lag}}}

\newcommand{\AR}{\textsc{AR}}
\newcommand{\MA}{\textsc{MA}}
\newcommand{\VAR}{\textsc{var}}
\newcommand{\VARX}{\textsc{varx}}
\newcommand{\VARMA}{\textsc{varma}}

\newcommand{\nmin}{n_{\textsc{min}}}
\newcommand{\GO}{\begin{bmatrix}\boldsymbol{G}^o_{:, 0}\\ \boldsymbol{G}_{\perp}\end{bmatrix}}
\DeclareMathOperator{\diag}{diag}
\DeclareMathOperator{\Tr}{Tr}
\DeclareMathOperator{\Mat}{Mat}

\title{A theorem of Kalman and minimal state-space realization of Vector Autoregressive Models.}
\author{Du Nguyen \\ nguyendu@post.havard.edu}

\begin{document}
\maketitle
\abstract{We introduce a concept of $autoregressive$ ($\AR$) state-space realization that could be applied to all transfer functions $\bT(L)$ with $\bT(0)$ invertible. We show that a theorem of Kalman implies each Vector Autoregressive model (with exogenous variables) has a minimal $\AR$-state-space realization of form $\by_t = \sum_{i=1}^p\bH\bF^{i-1}\bG\bx_{t-i}+\bep_t$ where $\bF$ is a nilpotent Jordan matrix and $\bH, \bG$ satisfy certain rank conditions. The case $\VARX(1)$ corresponds to reduced-rank regression. Similar to that case, for a fixed Jordan form $\bF$, $\bH$ could be estimated by least square as a function of $\bG$. The likelihood function is a determinant ratio generalizing the Rayleigh quotient. It is unchanged if $\bG$ is replaced by $\bS\bG$ for an invertible matrix $\bS$ commuting with $\bF$. Using this invariant property, the search space for maximum likelihood estimate could be constrained to equivalent classes of matrices satisfying a number of orthogonal relations, extending the results in reduced-rank analysis. Our results could be considered a multi-lag canonical-correlation-analysis. The method considered here provides a solution in the general case to the polynomial product regression model in \parencite{VeluReinselWichern}. We provide estimation examples with simulated data. We also explore how the estimates vary with different Jordan matrix configurations and discuss methods to select a configuration. Our approach could provide an important dimensional reduction technique with potential applications in time series analysis and linear system identification. In the appendix, we link the reduced configuration space of $\bG$ with a geometric object called a vector bundle.}
\section{Introduction}
Traditionally, the state-space approach to time series considers a representation:
$$\by_t = \bK \bv_t$$
$$\bv_{t} = \bC \bv_{t-1}+\bD \epsilon_t$$
Using the lag operator $L$, $\bv_t = (\bI-\bC L)^{-1}\bD\bep_t$ and thus 
\begin{equation}
\label{eq:state0}
\by_t = \bK(\bI-\bC L)^{-1}\bD\bep_t
\end{equation}
Let $\tilde{\bT}$ be a rational matrix function such that $\tilde{\bT}(\infty) = 0$ ($\tilde{\bT}$ is called strictly-proper in this case.) A realization is a representation of $\tilde{\bT}$ in the form
$$\tilde{\bT}(z) = \bK(z\bI  - \bC )^{-1}\bD$$
It is known \parencite{Gilbert, Kalman} a realization exists for all strictly-proper $\tilde{\bT}$. If $\by_t = \bT(L)\bep_t$ for a rational matrix $\bT(L)$ with $0$ is not a pole of $\bT$ ($\bT(0)$ is finite) then $\tilde{\bT}(L)= L^{-1}\bT(L^{-1})$ is strictly proper. Hence:
$$\bT(L) = L^{-1} \bK(L^{-1}\bI-\bC )^{-1}\bD = \bK(\bI-\bC L)^{-1}\bD$$
and so $\by$ can be represented in state-space form. ($\bT(0)= \bI$ in many models, structure models could have $\bT(0)\neq\bI$.)
We note, this traditional state-space realization (which we will call $\MA$-state-space realization) gives a representation of $\by$ in term of $\bep$. $\bC$ gives valuable information, for example its eigenvalues could determine stability of the process. As a moving average representation, it does not link $\by_t$ with its lagged values directly. We will take a different approach in this paper.

If $\bT(0)$ is invertible, we note $\bZ(s) = \bI - \bT(0)\bT(s^{-1})^{-1}$ is strictly proper as a function of $s$. We can apply the same realization theorem to express
$$\bZ(s) =\bH(s\bI-\bF)^{-1} \bG$$
for some $\{\bH, \bF, \bG\}$. With $s=L^{-1}$, this implies:
$$\bT(0)\bT(L)^{-1} = \bI - \bZ(L^{-1}) = \bI -\bH(L^{-1}\bI-\bF )^{-1}\bG = \bI - \bH(\bI -\bF L)^{-1}\bG L$$
And the model
$$\by_t = \bT(L)\bep_t$$
could be written as
$$\bT(L)^{-1}\by_t = \bep_t$$
or
$$\bT(0)\bT(L)^{-1}\by_t = (\bI - \bH(\bI - \bF L)^{-1}L)\by_t = \bT(0)\bep_t$$
$$\by_t = \bH(\bI -\bF L)^{-1}\bG L \by_t  + \bT(0)\bep_t$$
This is what we call the $autoregressive$ ($\AR$) state-space form. $\by$ could be forecast by its lagged values. This is an important feature we would like to explore in this paper. Consider the Vector Autoregressive ($\VAR$) model:
\begin{equation}
\label{eq:VAR}
\by_t = \sum_{i=1}^{p} \bPhi_i \by_{t-i} + \bep_t
\end{equation}
In this case, $\bT(L)^{-1} = \bI - \sum_{i=1}^{p} \bPhi_i L^{i}$. So $\bZ(s) = \sum_{i=1}^{p} \bPhi_i s^{-i}$. It is clear that $\bZ(s)$ is strictly proper. Moreover, it has only one pole of degree $p$ at $0$. Kalman described its minimal realization explicitly. We will see $\bF$ could be made a nilpotent Jordan matrix, and so could be classified by the shape of the Jordan blocks. For a Jordan form $\bF$ with $\bF^p=0$, this shows:
$$\bP(L) := \bZ(L^{-1}) = \sum_{i=1}^p \bPhi_i L^i = \sum_{i=1}^p (\bH\bF^{i-1} \bG)L^i$$
and the regression is
$$\by_t = \sum_{i=1}^p (\bH\bF^{i-1} \bG)L^i\by_{t-i} + \bep_t$$
Here, $\bF$ does not determine stability of $\by_t$, but $\by_t$ is explicitly expressed in its lagged values. This approach offers a number of crucial advantages. It turns out to be a generalization of the reduced-rank regression approach. In that case $\bF=0$, and we have $\by_t = \bH\bG\by_{t-1}+\bep_t$. We show that we can replicate most of the reduced-rank analysis here. Fixing $\bG$, $\bH$ could be computed by least square. The likelihood function could be expressed via the Schur determinant formula as a determinant ratio which could be considered a generalized Rayleigh quotient. This generalizes the classical result that reduced-rank $\VAR(1)$ models are related to generalized invariant subspace representations, and to the associated Rayleigh quotients. Therefore, maximizing the likelihood means minimizing a determinant ratio. The gradient and hessian of the likelihood function are very easy to compute, and could be used to estimated model parameters using standard optimizers in the examples we consider. However, similar to the reduced rank case, the likelihood function is unchanged if we replace $\bG$ by $\bS\bG$ if $\bS$ commutes with $\bF$. So it is possible to restrict the search space to a lower dimension set. In the reduced-rank case, using $QR$ factorization on $\beta'$ we can assume the rows of $\bG$ to be orthonormal. We have a similar situation in the minimal state-space case.

As the structure of $\bF$ could be classified by listing all Jordan forms with $\bF^p=0$, we have a very explicit and simple classification of possible realizations. The approach offers a systematic parameter reduction technique that we hope to compare and combine with other parameter estimation techniques.

Reduced-rank regressions could be defined for any two variables $\bx$ and $\by$, not only for an autoregressive $\by$. Our results are valid for a more general forecasting model with time lagged regressors, the $\VARX$ model. We restrict ourselves to consider $\VAR$ and $\VARX$ in this article. The more general case of $\VARMA$ will be considered in a future article.

We collect the symbols used and compares our minimal $\AR$-state-space approach with reduced-rank regression for the reader's convenience in \cref{tab:summarize}. The concepts and symbols will be introduced in subsequent sections.
\begin{table}[H]
\begin{tabular}{|p{3.5cm} | p{4.5cm}| p{7cm}|} 
 \hline
 Concept/Symbol &    Reduced-Rank & $\AR$-state-space  \\ [0.5ex] 
 \hline\hline
Dimension of $\bx$ &  $m$  &    $m$\\
Dimension of $\by$ &  $k$  &    $k$\\
$\min(m, k)$ & $h$ & $h$ \\
Lag & $p=1$ or not applicable & $p$\\ 
Structure params & reduced rank $\fll = d < m$ & $\hat{\Psi} = [d_1,\cdots,d_p], d_i \geq 0; d_p > 0; \sum d_i\leq h$ \\
Total rank alloc. & $\fll=d$ & $\fll = \sum d_i$ \\
Min. state-space dim. & $\fll=d$ & $\sum jd_j$ \\
Alt. struct. params & $\fll=d$ & $\Psi = [(r_g, l_g), \cdots,(r_1, l_1) )]_{r_g > \cdots > r_1, 0 < l_i = d_{r_i}}$ \\
Parameter reduction & $(m-d)(k-d)$ & $\sum_{i=1}^p(m-\sum_{j\geq i}d_j) \sum_{i=1}^p(k-\sum_{j\geq i}d_j)$\\
Jordan block & $\bJ(\lambda, r, l)$ & $\bJ(\lambda, r, l)$\\
$\lambda=0$ Jordan block & $\bK(r, l)$ & $\bK(r, l)$ \\
$\bF$ & $\bF= \bK(1, l) = 0_{d, d}$ & $\bF_{\Psi} = \oplus_{i=g}^1 \bK(r_i, l_i)$\\
Factorization & $\beta = \bH\bG$ & $\bP(L) = \bZ(L^{-1}) = \bH(\bI -\bF L)^{-1} \bG$\\
Minimal criteria & $\bG, \bH$ of size $(d, m), (k, d)$, rank $d$  & $\bG_{r, 0}, \bH_{r, 0}$ of size $(d_{r}, m), (k, d_r)$, rank $d_r$. $\bG_{:, 0} = (\bG_{r, 0})_r$ and $\bH_{:, 0} = (\bH_{r, 0})_r$ are of full row and column rank\\
$LQ$\slash Gram-Schmidt & $\bG\bG' = \bI_{\fll}$  & $\bG_{:, 0}\bG_{:, 0}' = \bI_{\fll}$; $\bG_{r, l} \bG_{r_1, 0}' = 0$ for $l>\max(0, r-r_1-1)$\\
Num. mat. ($\bA$) &$\bX\bX' -\bX\bY'(\bY\bY')^{-1}\bY\bX'$ & $\bXLag\bXLag' -\bXLag\bY'(\bY\bY')^{-1}\bY\bXLag'$\\
Denom. mat. ($\bB$) & $\bX\bX'$ & $\bXLag\bXLag'$ \\
$\kappa$ & $\kappa(\bG) = \bG$ & defined in \cref{eq:kappa} \\
Neg Log Likelhood &$\log(\det(\bG\bA\bG'))-\log(\det(\bG\bB\bG'))$ & $\log(\det(\kappa(\bG)\bA\kappa(\bG)')-\log(\det(\kappa(\bG)\bB\kappa(\bG)')$\\
Gradient & $2\Tr(\bG\bA\bG')^{-1}\bG\bA\eta'-2\Tr(\bG\bB\bG')^{-1}\bG\bB\eta'$ & $2\Tr(\kappa(\bG)\bA\kappa(\bG)')^{-1}\kappa(\bG)\bA\kappa(\eta)'-2\Tr(\kappa(\bG)\bB\kappa(\bG)')^{-1}\kappa(\bG)\bB\kappa(\eta)'$ \\
($\cH(\bA, \bG, \psi, \eta)$) & $2\Tr((\bG\bA\bG')^{-1}(\bG\bA\psi' + \psi\bA\bG') - (\bG\bA\bG')^{-1}\psi\bA\eta')$ &
$2\Tr((\kappa(\bG)\bA\kappa(\bG)')^{-1}(\kappa(\bG)\bA\kappa(\psi)' + \kappa(\psi)\bA\kappa(\bG)') - (\kappa(\bG)\bA\kappa(\bG)')^{-1}\kappa(\psi)\bA\kappa(\eta)')$ \\
Hessian & $\cH(\bA, \bG, \psi, \eta) - \cH(\bB, \bG, \psi, \eta)$ & $\cH(\bA, \bG, \psi, \eta) - \cH(\bB, \bG, \psi, \eta)$ \\
No. configs for $\bF$ & $h$ & $\begin{pmatrix}p+h-1\\ p\end{pmatrix}$\\
\hline
\end{tabular}
\caption{Main symbols and concepts.}
\label{tab:summarize}
\end{table}

To summarize, in this paper, we:
\begin{itemize}
\item Introduce a framework for dimensional reduction for $\VARX$ models under the concept of minimal $\AR$-state-space. We show in this case, minimal state-space could be classified explicitly in term of Jordan forms.
\item Compute the likelihood function for each configuration of Jordan form. It could be considered as a multi-lag canonical correlation analysis. While we can no longer maximize the likelihood via eigenvectors, its gradient and hessian are simple to compute, so we can apply standard optimization techniques.
\item Reduce the search domain for the likelihood function to a lower dimensional set, via a generalized $LQ$\slash Gram-Schmidt procedure. $\bG$ could be made to satisfy two sets of orthogonal relations. This allows us to apply manifold optimization techniques for large scale problems.
\end{itemize}
In the appendix, we show the reduced search domain could be considered a vector bundle on a flag manifold. This geometric concept is not essential to use our model but potentially will be helpful for higher dimension\slash autoregressive order.

The work \parencite{VeluReinselWichern} (which mentioned the model in \parencite{Brillinger:1969}) is probably a predecessor to this approach. It is related to the case $\bF= \bK(p, d_p)$ (Jordan block of one exponent). The author handled the case where $\bH_{p, j} = 0$ for all $j > 0$. We will discuss this model in more details in \cref{sec:examples}. We will show our approach can also be used for their model in the general case.

As we only consider $\AR$-state-space models in this article, we often drop the prefix $\AR$ when mentioning state-space models. We will use the wild card character {\it :} to replace running indices on (block) rows and columns of matrices.
\section{Review of reduced-rank regression and VAR(1) minimal realization.}
Reduced-rank regression was first studied in \parencite{Anderson}. It found applications in different areas of statistics, notably in time series. Johansen \parencite{Johansen} used it in his famous test of cointegration. Reduced-rank regression for time series has been studied by \parencite{VeluReinselWichern, AhnReinsel, Anderson1999, Anderson2002}. \parencite{BoxTiao} introduced canonical-correlation-analysis to time series. We review reduced-rank regression briefly here, in a less general framework but sufficient for the subsequent analysis. We study a model of form:
$$\by = \beta \bx + \bep$$
$\by$ is a $k$-dimensional random variable, $\bx$ is a $m$ dimension variable and $\beta$ is a $k\times m$ matrix of rank $d\leq \min(k, m)$. We can set $\beta = \bH\bG$ with $\bH$ of size $k\times d$ of rank $d$ and $\bG$ of size $d\times m$. Given sample matrices $\bY$ of size $k\times n$ and $\bX$ of size $m\times n$; for a fixed $\bG$, the optimal $\bH$ is obtained by:
$$\bH = \bY\bX^{\prime}\bG'(\bG\bX\bX'\bG')^{-1}$$
and the residual covariance matrix is
$$\cC(\bG)= \bY\bY' - \bY\bX'\bG'(\bG\bX\bX'\bG')^{-1}\bG\bX\bY'$$
Following \parencite{Johansen_book}, we apply the Schur determinant formula to the block matrix:
$$\begin{pmatrix}
\bY\bY' & \bY\bX'\bG'\\
\bG \bX \bY' & \bG \bX\bX' \bG'
\end{pmatrix}$$
We have:
$$\begin{aligned}\det(\bY \bY') \det(\bG \bX\bX' \bG' - \bG\bX\bY'(\bY\bY')^{-1}\bY\bX'\bG') =\\
\det(\bG \bX\bX' \bG')\det(\bY \bY' - \bY\bX'\bG'(\bG \bX\bX' \bG')^{-1}\bG \bX \bY')\end{aligned}$$
So to minimize $\det(\cC(\bG))$ (as a function of $\bG$) we need to minimize the ratio:
$$\cR(\bG) = \frac{\det(\bG [\bX\bX'  - \bX\bY'(\bY\bY')^{-1}\bY\bX']\bG')}{\det(\bG \bX\bX' \bG')}$$
or its logarithm, which has a simple gradient:
$$\nabla_{\eta}D\log(\cR) = 2\Tr((\bG\bA\bG')^{-1}\bG\bA\eta' - (\bG\bB\bG')^{-1}\bG\bB\eta')$$
Where $\bA = [\bX\bX'  - \bX\bY'(\bY\bY')^{-1}\bY\bX']$ and $\bB = \bX\bX'$. Here $\nabla_{\eta}$ is the directional derivative in the direction $\eta$. At a critical point we have
$$(\bG\bA\bG')^{-1}\bG\bA = (\bG\bB\bG')^{-1}\bG\bB$$
Therefore, $\bG\bA = \gamma \bG\bB$, where $\gamma =(\bG\bA\bG')(\bG\bB\bG')^{-1}$ is a matrix of size $d\times d$. So this is a generalized invariant subspace problem. Rewrite it as:
$$\bG\bB^{1/2}\bB^{-1/2}\bA\bB^{-1/2} = \gamma \bG\bB^{1/2}$$
This becomes an invariant subspace problem, where the new matrix is $\bB^{-1/2}\bA\bB^{-1/2}$ and $\bG\bB^{1/2}$ is the new variable. Alternatively, by comparing gradient it is well-known that this determinant ratio minimizing problem is equivalent to the trace ratio problem:
$$ \Tr((\bG [\bX\bX'-\bX\bY'(\bY\bY')^{-1}\bY\bX']\bG')(\bG \bX\bX' \bG')^{-1})$$
and this leads to the problem of maximizing 
$$\Tr((\bG [\bX\bY'(\bY\bY')^{-1}\bY\bX']\bG')(\bG \bX\bX' \bG')^{-1})$$
which brings us to canonical-correlation-analysis.

In the time series case, we have $\by=\by_t$ and $\bx = \by_{t-1}$. The corresponding regression is:
$$\by_t = \bPhi \by_{t-1} + \bep_t$$
This is the vector autoregressive model $\VAR(1)$, with $p=1$ as we only have one lag. If $\bPhi$ is of reduced-rank $d$: $\bPhi = \bH\bG$ as above, which can be written with lag operator symbol as
$$(\bI_k-\bPhi L)\by_t = \bep_t$$
Its transfer function, is $\bT(L)=(\bI-\bPhi L)^{-1}$, so as in the introduction $\bZ(s)=\bI-\bT(s^{-1})^{-1}=s^{-1}\bPhi $ has minimal state-space form:
$$\bZ(s) = \bH[s \bI_r ]^{-1}\bG$$
From our discussion so far, it is clear that the minimal $\AR$-state-space realization of the lag polynomial of $\VAR(1)$ model is exactly the reduced-rank regression model.

\section{Minimal state-space realization for VAR(2)}
Let us consider the case $p=2$.  In this case, $\bZ(s) = \bPhi_1 s^{-1} + \bPhi_2 s^{-2}$ for the regression:
$$\by_t = \bPhi_1 \by_{t-1}+ \bPhi_2 \by_{t-2} + \bep_t$$
With $s = L^{-1}$, we want to realize $\bP(L) = \bPhi_1 L + \bPhi_2 L^2$ in the form $\bH(L^{-1}\bI - \bF)^{-1}\bG = \bH(\bI - \bF L)^{-1}\bG L$. The later expression will need to be a polynomial in $L$. We note if $\bF$ is nilpotent with  $\bF^2=0$ then $(\bI-\bF L)^{-1} = \bI+\bF L$, therefore:
$$\bPhi_1 = \bH \bG$$
$$\bPhi_2 = \bH\bF\bG$$
Assume that is the case, we can assume further that $\bF$  is of Jordan form:
$$\bF = \begin{bmatrix}0_{l_2, l_2} & \bI_{l_2} & 0\\ & 0_{l_2, l_2} & 0\\ & &0_{l_1, l_1} \end{bmatrix}$$
We can divide $\bH$ and $\bG$ to corresponding blocks, $\bH = \begin{bmatrix}\bH_{2, 0} & \bH_{2, 1} &  \bH_{1, 0}\end{bmatrix}$, $\bG=\begin{bmatrix} \bG_{2, 1}\\ \bG_{2, 0}\\ \bG_{1, 0}\end{bmatrix}$. Expanding:
\begin{equation}
\label{eq:Phi2}
\begin{gathered}
\bPhi_1 = \bH_{1, 0} \bG_{1, 0} + \bH_{2, 0} \bG_{2,1} + \bH_{2, 1} \bG_{2, 0}\\
\bPhi_2 = \bH_{2, 0} \bG_{2, 0}
\end{gathered}
\end{equation}
Therefore, we have a realization if we can decompose $\bPhi_1$ and $\bPhi_2$ to this form. The $minimal$ requirement would put further restrictions: if $\bG_{2, 0}$ and $\bG_{1, 0}$ has zero rows, then we can just drop those rows and get a smaller state-space realization. So a starting condition is $\bG_{2, 0}$ and $\bG_{1, 0}$ should not have zero rows. The actual condition, specified by Kalman is that the rows of $\bG_{2, 0}$ and $\bG_{1, 0}$ are linearly independent, and so are the columns of $\bH_{2, 0}$ and $\bH_{1, 0}$. This puts a constraint $l_1 + l_2 \leq k$ where $k$ is the dimension of the vector $\by_t$. Further, he proved our guess is correct: $\bF$ needs to be a nilpotent matrix. The discussion here could be generalized to $\VAR(p)$ models and more generally to regression models with time lag structures as we can see in the next sections.

\section{A result of Kalman on rational matrix function}
The results in this section is purely algebraic involving matrix polynomial and rational functions. Its main result was discovered by Kalman in \parencite{Kalman}. Together with earlier results of Gilbert in \parencite{Gilbert} they give a complete picture of minimal state-space realization of all proper rational transfer functions. We will recall a few definitions but will mostly focus on Proposition 3 of \parencite{Kalman}, which is most relevant to our situation. As before, a rational matrix function $\bZ(s)$ is called strictly proper if $\lim_{s\to\infty}\bZ(s) = 0$. By factoring out the (scalar) least common denominator $\baq(s)$ we can write:
\[ \bZ(s) = \frac{1}{\baq(s)} \baP(s)\]
with degree of $\baP$ is less than degree of $\baq$. Gilbert addressed the case where $\baq$ has simple roots. In that case we can expand $\bZ$ by partial fraction:
\[ \bZ(s) = \sum_{i=1}^{g} \frac{\bH_i \bG_i}{s - \lambda_i} \]
with $\bH_i, \bG_i$ are of sizes $k\times d_i, d_i\times k$ and of rank $d_i$. Note we are assuming that $\baq(s)$ could be factored to monomials, hence $\bH_i, \bG_i$ and $\lambda_i$ could be complex numbers (but in the end $\bZ$ is real.) Gilbert showed that $\bZ$ admits a minimal state-space realization with minimal state $\sum_i d_i$ and constructed it as a direct sum of blocks of form $\bH_i(s \bI_{r_i}-\lambda_i)^{-1}\bG_i$. The case $g=1$ and $\lambda_i = 0$ is the $\VAR(1)$ case above, with $L=s^{-1}$ as usual.

Kalman addressed the case of roots with multiplicity. He showed in general the minimal state-space realization could be constructed as direct sum of realization for distinct roots, each could have multiplicity greater than $1$. For the case of one root his result could be summarized in the following proposition, which is a restatement of Proposition 3 of \parencite{Kalman} which works for matrices in $\cX = \R^n$ or $\C^n$ as this is purely an algebraic result. We note he used the term $irreducible$ instead of $minimal$ realization, which is the standard term today. For our application the root will be zero and all coefficients will turn out to be real.
\begin{proposition}
Let $\bZ(s) = \frac{1}{(s-\lambda)^p} \baP(s)$ with $\baP(s)$ is a polynomial matrix of degree less than $p$. Let $\bF$ be an $n\times n$ matrix with a single eigenvalue $\lambda$. We take a basis of $\bF$ in $\cX$ so that $\bF$ has the Jordan form:
$\bF = \diag(\bJ(\lambda, r_g, l_g), \cdots, \bJ(\lambda, r_1, l_1))$ with $r_g > r_1 > \cdots > r_1$. Where $\bJ(\lambda, r, l)$ is defined by:
$$\bJ(\lambda, r) =\begin{bmatrix}\lambda & 1&  & & & \\
 & \cdots & \cdots & & & \\
 & & \cdots & \cdots & &\\
 & &  & \cdots & &1\\
 0 & & & & & \lambda
\end{bmatrix}$$
$$\bJ(\lambda, r, l) = \bJ(\lambda, r) \otimes \bI_l$$
and $\bJ(\lambda, r)$ is of size $r$. Let $\bG, \bH$ be $n\times m, k\times n$ matrices expressed with respect to the same basis. Let $\fll = \sum l_i$. Then $\{\bH, \bF, \bG\}$ is a minimal state-space realization of $\bZ(s)$ if and only if both of the following conditions are satisfied:
\begin{enumerate}
\item The $\fll\times m$ matrix $\bG_{:, 0}$, which consists of rows $(r_g-1)l_g+1, \cdots, r_g l_g, r_g l_g + (r_{g-1}-1)l_{g-1}+1,\cdots, r_g l_g + r_{g-1}l_{g-1},
\cdots, \sum_2^{g}r_i l_i +(r_1-1)l_1+1 ,\cdots, \sum_1^{g}r_i l_i$ has rank $\fll$.
\item The $k\times \fll$ matrix $\bH_{:, 0}$, which consists of columns $1,\cdots, l_g,  r_g l_g+1, \cdots r_g l_g + l_{g-1},\cdots, \sum_2^{g}r_i l_i +1 ,\cdots, \sum_2^{g}r_i l_i + l_1$ has rank $\fll$.
\end{enumerate}
\end{proposition}
We have mostly preserved Kalman's notations, the notable differences are:
\begin{enumerate}
\item Replacing his $p$ with $k$ so not to confuse with the degree of the $\VAR$ model.
\item Use $g$ for the number of terms distinct $r$'s, as $q$ may be confused with moving average order.
\item Grouping the blocks with the same $r$ together and order the blocks in descending order of $r$'s. This is to conform with the order of the McMillan denominator. We will see the block with exponent $r$ will contribute to $r$ coefficients $\bPhi_{1},\cdots, \bPhi_r$, and the reduction in the rank of the $\bG_{p, 0}$ block contribute the most to reduction of overall parameters. So in a sense it is an order of importance.
\end{enumerate}
We note the dimension of the minimal realization is
$$\nmin= \sum r_i l_i = \delta_M$$
$\delta_M$ is defined in term of the Smith normal form. If a minimal state-space realization is given, $\bZ(s)$ is recovered by expanding the terms. Conversely, given $\bZ(s)$, Kalman gave an algorithm to recover $\{\bH, \bF, \bG\}$. The algorithm is based on representing $\baP(s)$ in Smith normal form then expand the terms to Taylor series of terms $s-\lambda$ and read the coefficients off from the representation. Kalman's proposition is essentially a translation between the Smith-McMillan form and the state-space form, a point carried out to all base fields by the work of \parencite{ItoWimmer}.

Our interest is in the case $\lambda = 0$. We will use the notation $\bK(r, l) = \bJ(0, r, l)$ going forward. Consider again
$$\bZ(s) = \bPhi_1 s^{-1} +\cdots \bPhi_p s^{-p} = \frac{1}{s^p}\sum_{i=1}^p \bPhi_{p-i} s^{i}$$
with $\bPhi_p \neq 0$, let $\bP(L)=\bZ(L^{-1})$:
$$\bP(L) = \bPhi_1 L +\cdots \bPhi_pL^p$$
Applying the proposition for this case:
$$\bZ(s) = \bH(s \bI -\bF)^{-1}\bG = \bH(\bI-\bF s^{-1})^{-1}s^{-1}\bG$$
and observe since $\lambda=0$, $\bF$ is a $nilpotent$ Jordan matrix: $\bF^p = 0$. Therefore we could rewrite the state-space realization in term of $L$:
\begin{equation}
\bP(L) = \bZ(L^{-1}) = \bH (\bI - \bF L)^{-1}\bG L = \bH\bG L + \bH\bF\bG L + \cdots \bH\bF^{p-1}\bG L^p
\end{equation}
or
$$\bPhi_i = \bH\bF^{i-1}\bG$$
From the fact $\bF^{r_g} = 0$ and $\bPhi_p \neq 0$ this implies $r_g = p$.

\section{Detailed description of the minimal realization.}
Let us now go deeper to the structure of $\bG$ and $\bH$. Since $\bF$ is composed of $g$ blocks $\bK(r_i, l_i)$, $\bG$ and $\bH$ could be decomposed to the corresponding $g$ row or column blocks respectively. The block corresponding to $r_i$ is of size $r_i l_i$ and it has $r_i$ sub-blocks, each of $equal$ size $l_i$. We call $r$ the exponent of the Jordan block $\bK(r, l)$ and $l$ the sub-rank.

We index the sub-blocks corresponding to an exponent $r$ of  $\bG$ in descending order: $\bG_{r, r-1}, \cdots,  \bG_{r, 0}$. The corresponding sub-blocks in $\bH$ are in ascending order, $\bH_{r, 0}, \cdots, \bH_{r, r-1}$. We will see this more explicitly in the next section. The somewhat mysterious indexing has origin from the correspondence between the Smith normal form and state-space realization. With this double indexing convention, it is clear $\bG_{:, 0}$ above is just the collection of all $\bG_{r_i, 0}$, and the assumption is the rows of $\bG_{:, 0}$ are linearly independent of rank $\fll = \sum l_i \leq \min(k, m)$. We have a similar observation for the sub-blocks of $\bH$. 

Let $h = \min(k, m)$. Instead of describing the Jordan blocks by pairs $(r_i, l_i)$ with $l_i > 0$, it is sometime more convenient to allow zero Jordan blocks. To summarize:
\begin{itemize}
\item For a $\VAR$ model the possible choices of $\bF$ could be classified by Jordan matrices such that $\bF^{p-1}\neq 0; \bF^p=0$. In other words, it could be classified by a list of tuples $\Psi = [(r_g, l_g), \cdots, (r_1, l_1)]$ with $l_i> 0, r_g = p, r_g > r_{g-1}> \cdots > r_1$, together with the rank constraint: $\fll = \sum l_i \leq h$. The corresponding Jordan matrix is $\bF = \bF_{\Psi}=\oplus_{i=g}^{i=1} \bK(r_i, l_i)$.
\item An alternative classification is by nonnegative integer value vector $\hat{\Psi} = [d_1,\cdots, d_p]$ with $d_p > 0$, $0\leq d_i$ and $\fll = \sum d_p \leq h$. $\Psi$ is obtained from $\hat{\Psi}$ as the list of tuples $[(i, d_i) | d_i > 0]$ in reversed order of $i$. Conversely, we can obtain $\hat{\Psi}$ from $\Psi$ by patching $d_i=0$ for the exponents $i$ not in $\Psi$. We call $\Psi$ and $\hat{\Psi}$ {\it structure parameters}.
\item To be consistent with the convention of the McMillan denominator, we will write the Jordan blocks in descending order of exponent.
\item For each $1 \leq i \leq p$, if $d_i > 0$, a Jordan block is defined by its exponent $i$ and its sub-rank $d_i$. $\bK(i, d_i)= \bJ(0, i, d_i)$ is of size $d_ii$. We will skip blocks with $d_i=0$.
\item The minimal state-space dimension, which is the dimension of $\bF$ in a minimal realization is $\nmin=\sum_{i=1}^p d_ii = \sum_{j=g}^{1} r_j l_j$. It is equal to the McMillan degree.
\item For a given tuple $(k, m, p)$, and $h=\min(k, m)$, the highest possible minimal state-space dimension is $ph$ and corresponds to $\Psi = [(p, h)]$. The lowest possible dimension is $p$, corresponds to $\Psi=[(p, 1)]$.
\item The number of possible configurations of $\Psi$ with maximal degree $p$ is $\begin{pmatrix} h + p-1 \\ p \end{pmatrix}$. This follows for a balls-and-urns computation. A more straight forward application of balls-and-urns is the number of configurations with $\fll = \sum d_i \leq h$ and degree {\it not exceeding} $p$: in addition to the urns $1\cdots p$, we add an urn corresponding to unused dimension $d_0 = h -\fll$. This is an $h$ balls, $p+1$ urns problem with number of configurations $\begin{pmatrix} h + p \\ p \end{pmatrix}$. Our count for the maximal degree is exactly $p$ comes from 
$$\begin{pmatrix} h + p -1  \\ p \end{pmatrix} = \begin{pmatrix} h + p \\ p \end{pmatrix} - \begin{pmatrix} h + p -1\\ p-1 \end{pmatrix}$$
In the code, we include a function to list all possible $\Psi$ per a pair $(h, p)$.
\item From here the number of configurations growths polynomially (of degree $p$) in $m$. There are $ph$ possible minimal state-space dimensions, and the distribution of number of $\Psi$'s per state-space dimension is a bell-shape. The analysis suggests that for large $m$ and $p$, iteration through the set of all possible $\Psi$ is not practical. We will discuss estimation in the next section.
\end{itemize}

As an example, for $p =3$, $\bP(L)=\sum \Phi_i L^i$ is represented as:
\[  L\begin{bmatrix}\bH_{[3, :]} & \bH_{[2, :]} & \bH_{[1, :]} \end{bmatrix} \begin{bmatrix}\bI - \bK(3, d_3)L &  & \\
 & \bI - \bK(2, d_2)L & \\
 &  & \bI -\bK(1, d_1)L\\
\end{bmatrix}^{-1}
\begin{bmatrix}\bG_{[3, :]} \\ \bG_{[2, :]} \\ \bG_{[1, :]} \end{bmatrix}
\]
As before we can use the nilpotency of $\bK(r, l)$:
$$(\bI- \bK(r, l)L)^{-1} = \bI + \sum_{i=1}^{r-1} \bK(r, l)^iL^i$$
We note for the block $\bK(r, l)^i$ is the matrix with the $i^{th}$ offset upper block diagonal is $\bI_l$, and other entries are zero. Therefore, the contribution of that block is
\begin{equation}
\begin{bmatrix}\bH_{r, 0}, \cdots \bH_{r, r-1}\end{bmatrix}\begin{bmatrix}
\bI_l L & \bI_l L^2 & \bI_l L^3\cdots  &\cdots & &\bI_l L^r \\
 & \cdots & \cdots & & \bI_l L^{r-1} & \\
 & & \cdots & \cdots & &\vdots\\
 & &  & \cdots & &\bI_l L^2\\
 0 & & & & & \bI_lL
\end{bmatrix}\begin{bmatrix} \bG_{r, r-1} \\ \vdots \\ \bG_{r, 0}\end{bmatrix}
\end{equation}
 So its contribution to $\bPhi_i = \bH\bF^{i-1} \bG$ is:
 \[
\sum_{a=0}^{r-i} \bH_{r, a} \bG_{r, p-i-a}
\]
And so:
\begin{equation}
\label{eq:Phi}
\bPhi_i = \sum_{j\geq i} \sum_{a=0}^{j-i} \bH_{j,a} \bG_{j, j-i-a}
\end{equation}
It is now simple to recover the formula for the cases $p=1$ and $p=2$ in previous sections. While the set up looks more involved to remember we have the following rules:
\begin{itemize}
\item $\bH_{j_1, a}\bG_{j_2, b}$ only appears if $j_1 = j_2$. So only terms of form $\bH_{j, a}\bG_{j, b}$ are present. There are no terms linking distinct Jordan blocks.
\item $\bH_{j, a}\bG_{j, b}$ contributes to $\bPhi_{j - a - b}$
\end{itemize}

In the next section we will describe the associated regression model, as well as estimations.

\section{VARX and least square estimates.}
We will describe a model of the form:
\begin{equation}
\label{eq:VARX}
\by_t = \bPhi_1 \bx_{t-1} + \bPhi_2 \bx_{t-2} \cdots +\bPhi_p \bx_{t-p} + \bep_t
\end{equation}
We allow $\by_{t-i}$ to be part of $\bx_{t-i}$, as discussed in \parencite{VeluReinselWichern}. The classical reduced-rank regression would be a special case of \cref{eq:VARX} with $p=1$.

Let us now turn to the estimation problem. Assuming we have $T+p$ samples of data, we organize the data in to matrices $\bY_f$ of size $k\times (T+p)$ and $\bX_f$ of size $m\times (T+p)$. Let $\bY$ be the submatrix of $\bY_f$ of size $k\times T$ skipping the first $p$ (columns) samples. Define $L^i\bX$ to be the submatrix of $\bX_f$ of size $m\times T$ skipping the last $i$ samples and the first $p-i$ samples. As we will not be using the first $p$ samples of $\bY_f$ as well as the last sample of $\bX_f$, they are allowed to be null. In the situation where we have autoregression it may be advantageous to share storage of $\bX_f$ and $\bY_f$. However we will consider $\bX_f$ and $\bY_f$ to be in separate matrices  here to simplify the notations. We now look at the problem of estimating $\bPhi_i$ such that the minimal state-space realization of $\bZ(s) = \bPhi_1 s^{-1} + \bPhi_2 s^{-2} \cdots +\bPhi_p s^{-p}$ corresponds to $\hat{\Psi} = \{d_1, \cdots d_p\}$. 

As with classical regressions, the maximum likelihood estimate with Gaussian noise of the parameters $\bH, \bG, \Omega$
will have the form:
$$-\frac{T}{2}\log(\det(\Omega)) - \frac{1}{2}\sum^T(\by_t - \sum\bPhi_i \bx_{t-i})'\Omega^{-1} (\by_t - \sum\bPhi_i \bx_{t-i})$$
where $\bPhi_i$ is given by \cref{eq:Phi}. We arrive at the condition:
$$\Omega = \frac{1}{T}(\bY - \sum\bPhi_i L^i\bX)(\bY - \sum\bPhi_i \bX)'$$

Since $\bF$ is known from the specification of $\hat{\Psi}$, similar to the reduced-rank case
we will need to estimate $\bH$ and $\bG$. Before we proceed with the general case let us go through the case $p=2$. From \cref{eq:Phi2} we have:
$$\by_t = (\bH_{1, 0} \bG_{1, 0} + \bH_{2, 0} \bG_{2,1} + \bH_{2, 1} \bG_{2, 0}) \bx_{t-1} + \bH_{2, 0} \bG_{2, 0} \bx_{t-2} + \bep_t$$
Similar to the $\VAR(1)$ case, to find a least square estimate by minimizing the determinant of the covariance matrix of $\epsilon_t$, we fix $\bG$ and find the optimized $\bH$. With the regressor:
$$\begin{bmatrix}\bG_{2, 1} L\bX + \bG_{2, 0}L^2\bX \\ \bG_{2, 0}L\bX \\  \bG_{1, 0} L\bX \end{bmatrix}$$
We can write it as $\kappa(\bG) \Xlag$, where:
$$\Xlag = \begin{bmatrix}L^2\bX \\ L\bX \end{bmatrix}$$
$$\kappa(\bG) = \begin{bmatrix}\bG_{2, 0} & \bG_{2, 1}   \\  & \bG_{2, 0}
 \\  & \bG_{1, 0} \end{bmatrix}$$
And note that we will write the exponents in descending order as is the case of the McMillan denominator. We see:
$$\bH(\bG) = \begin{bmatrix}\bH_{2, 0} & \bH_{2, 1} & \bH_{1, 0} \end{bmatrix} = \bY\Xlag'\kappa(\bG)'(\kappa(\bG)\Xlag\Xlag'\kappa(\bG)')^{-1}$$
So we need to minimize the determinant of residual covariance matrix:
$$\det(\bY\bY' - \bY\Xlag'\kappa(\bG)'(\kappa(\bG)\Xlag\Xlag'\kappa(\bG)')^{-1}\kappa(G)\Xlag \bY )$$
Again, using the Schur complement trick we need to minimize:
$$\cR(G) =  \frac{\det(\kappa(\bG) [\Xlag \Xlag'  - \Xlag \bY'(\bY\bY')^{-1}\bY\Xlag']\kappa(\bG)')}{\det(\kappa(\bG) \Xlag\Xlag' \kappa(\bG)')}$$
Write its logarithm as  $\log(\det(\kappa(\bG)\bA\kappa(\bG)' -\log(\det(\kappa(\bG)\bB\kappa(\bG)'))$ with $\bA = [\Xlag \Xlag'  - \Xlag \bY'(\bY\bY')^{-1}\bY\Xlag']$ and $\bB = \Xlag\Xlag'$ as before. The logarithm has a simple gradient by Jacobi's formula:
$$\nabla_{\eta}{\log(\det(\kappa (\bG)\bA\kappa(\bG)'))} = 2\Tr((\kappa (\bG)\bA\kappa(\bG)')^{-1} \kappa (\bG)\bA\kappa(\eta)'$$
where $\eta$ is a matrix in the shape of $\bG$ to specify the direction for the directional derivative $\nabla_{\eta}$.
$$\nabla_{\eta}\log(\cR(\bG)) = \Tr((\kappa (\bG)\bA\kappa(\bG)')^{-1} \kappa (\bG)\bA\kappa(\eta)'
- \Tr((\kappa (\bG)\bB\kappa(\bG)')^{-1} \kappa (\bG)\bB\kappa(\eta)' $$
for all $\eta$. So far, we can generalize the steps of the reduced-rank regression in the introduction. From here, the situation divergences. As $\kappa$ must be represented as a tensor, we do not have a matrix equation for $\bG$. However, since we know the gradient (and the hessian is also easy to compute) we can use a hessian-based optimizer. Later on we will see this is an optimization problem with the underlying function is invariant under a large group of matrix operations. We may use manifold optimization techniques for faster convergence. To conclude the section we note the analysis thus far generalizes to higher $p$:

\begin{theorem}
\label{th:likelihood}
Assume the minimal state-space realization of $\sum_{i=1}^p \bPhi_i L^{i}$ is represented by a nilpotent Jordan matrix consisting of blocks $\bK(p, d_p), \cdots,\bK(1,d_1)$. Let $\nmin = \sum d_ii$. For a fix $\bG$:
$$\bG = \begin{bmatrix}\bG_{p, p-1}\\ \vdots \\ \bG_{p, 0} \\ \vdots \\
\bG_{2, 1} \\ \bG_{2, 0} \\
\bG_{1, 0}
\end{bmatrix}$$
We define $\kappa(\bG)$ to be the block matrix with row blocks indexed by $(r, l)$ with $r$ arranged in descending order, $l$ arranged in in {\it ascending} order ($0\leq l\leq r-1$), and column blocks ordered from $p$ to $1$:
\begin{equation}
\label{eq:kappa2}
\kappa(\bG)_{(r, l), i} = \left\{\begin{array}{l} \bG_{r, r-l-i}\text{ if }i \leq r-l\\
0 \text{ otherwise.}
\end{array}
\right.
\end{equation}
In other words, the row blocks from $\kappa(\bG)_{r, r-1}$ to $\kappa(\bG)_{r, 0}$ has the right most column filled by $\bG_{r, r-1}$ to $\bG_{r, 0}$ from the top down, then the blocks are propagated up diagonally.
\begin{equation}
\label{eq:kappa}
\kappa(\bG) = \begin{bmatrix}
\bG_{p, 0} & \bG_{p, 1} & \vdots & \vdots & \bG_{p, p-2} & \bG_{p, p-1}\\
0 & \bG_{p, 0} & \bG_{p, 1} & \vdots & \vdots & \bG_{p, p-1}\\
 \vdots & \vdots & \vdots & \vdots & \vdots & \vdots \\
  0 & \cdots & \cdots & \cdots & 0 & \bG_{p, 0}\\
  \vdots & \vdots & \vdots & \vdots & \vdots & \vdots \\
  0 & \cdots & \cdots & 0 & \bG_{2, 0} & \bG_{2, 1}   \\
 0 & \cdots & \cdots & \cdots & 0 & \bG_{2, 0}\\
   0 & \cdots & \cdots & \cdots & 0 & \bG_{1, 0}\\
\end{bmatrix}
\begin{matrix}(p, 0)\\ (p, 1)\\ \vdots\\ (p, p-1)\\ \vdots \\ (2, 0) \\ (2, 1) \\ (1, 0)\end{matrix}
\end{equation}
Set 
\begin{equation} \Xlag = \begin{bmatrix}L^p\bX \\ \vdots \\ L\bX \end{bmatrix}
\end{equation}
Then the optimal $\bH$ to minimize the determinant of the residual covariance matrix is given by:
\begin{equation}
\begin{gathered}
\bH(\bG) = \begin{bmatrix}\bH_{p, 0}\cdots \bH_{p, p-1} & \cdots  & \bH_{2, 0} & \bH_{2, 1} & \bH_{1, 0} \end{bmatrix} = \\ \bY\Xlag'\kappa(\bG)'(\kappa(\bG)\Xlag\Xlag'\kappa(\bG)')^{-1}
\end{gathered}
\end{equation}
and the residual determinant is:
\begin{equation}
\label{eq:resi_covar}
\det(\bY\bY' - \bY\Xlag'\kappa(\bG)'(\kappa(\bG)\Xlag\Xlag'\kappa(\bG)')^{-1}\kappa(\bG)\Xlag \bY )
\end{equation}
$\kappa(\bG)$ is of full row rank if $\bG_{:, 0}$ is of full row rank. Minimizing \cref{eq:resi_covar} is equivalent to minimizing:
\begin{equation}
\label{eq:ratio_det}
\cR(\bG) = \frac{\det(\kappa(\bG) [\Xlag \Xlag'  - \Xlag \bY'(\bY\bY')^{-1}\bY\Xlag']\kappa(\bG)')}{\det(\kappa(\bG) \Xlag\Xlag' \kappa(\bG)')}
\end{equation}
which has the $\log$-gradient:
\begin{equation}
\begin{gathered}
 2\Tr((\kappa (\bG)\bA\kappa(\bG)')^{-1} \kappa (\bG)\bA \kappa(\eta)') \\
- 2\Tr((\kappa (\bG)\bB\kappa(\bG)')^{-1} \kappa (\bG)\kappa(\eta)')
\end{gathered}
\end{equation}
with $\bA =\Xlag \Xlag'  - \Xlag \bY'(\bY\bY')^{-1}\bY\Xlag'$ and $\bB = \Xlag\Xlag'$.
Define $\cH(\bA, \bG, \psi, \eta)$ to be
\begin{equation} 
2\Tr((\kappa(\bG)\bA\kappa(\bG)')^{-1}(\kappa(\bG)\bA\kappa(\psi)' + \kappa(\psi)\bA\kappa(\bG)') - (\kappa(\bG)\bA\kappa(\bG)')^{-1}\kappa(\psi)\bA\kappa(\eta)')
\end{equation}
Then its Hessian is $\cH(\bA, \bG, \psi, \eta) - \cH(\bB, \bG, \psi, \eta)$.
\end{theorem}
\begin{proof}
The proof is a generalization of the case $p=2$. We note as before, row blocks indices of $\kappa$ are ordered in order of $\bH$, so the row block indices are $(r, 0), \cdots (r, r-1)$; in opposite with convention for $\bG$. The columns are indexed from $p$ to $1$ as they correspond to $\Xlag$. First, we need to prove the regressor to $\bH$ is $\kappa(\bG)\Xlag$. From \cref{eq:Phi}, the block of regressor corresponding to $\bH_{r, l}$ is $\sum_i \bG_{r, r-l-i}L^{i}\bX$, but this is the $(r, l)$ row block of $\kappa(\bG)\Xlag$ by \cref{eq:kappa2}. Next, we need to show $\kappa$ is of full row rank. If $v\kappa(\bG) = 0$ then the columns of $v$ corresponding to blocks $(r, 0)$ are zeros, as rows of $\bG_{:, 0}$ are linearly independent. From the block triangular shape of $\kappa$ we can show $v_{r, l}=0$ inductively in $l$. Hence, $\kappa(\bG)\bU \kappa(\bG)'$ does not have zero as an eigenvalue if $\bU$ is positive definite. Therefore, $\kappa(\bG)\bU \kappa(\bG)'$ is also positive definite, so the determinant ratio and its logarithm in the theorem are well-defined. The remaining calculations are routine.
\end{proof}
In a sense, our determinant ratio likelihood could be considered as a multiple lag version of canonical-correlation-analysis, as the gradient equation reduces to the same calculation in the $p=1$ case.
\section{Equivalence of state-space realizations.}
\label{sec:equivalence}
In the simple reduced-rank case, \cref{eq:ratio_det} is unchanged when $\bG$ is replaced by $\bS\bG$ for any invertible matrix $\bS$. In our general framework, a similar result holds for an invertible matrix $\bS$ such that $\bS\bF = \bF\bS$. We will describe such matrices, and show we can normalize $\bG$ to a form satisfying certain orthogonal relations. This section is more on linear algebra and geometry, readers can skip the section if they are not interested in the details of dimensional reduction for estimation algorithms. The main result to take away is the parameter count of $\bS$, which implies the parameter reduction count for the model. On the other hand, it is not difficult to work out all the details of the section in the case $m=2$ or $m=3$ manually, and that would probably give readers more intuition on parameter reduction.

As pointed out by Kalman in the same paper, the realizations $\{\bH, \bF, \bG\}$ and $\{\bH\bS^{-1}, \bS\bF\bS^{-1}, \bS\bG\}$ are equivalent. If $\bS$ commutes with $\bF$, then it is equivalent to $\{ \bH\bS^{-1}, \bF, \bS\bG\}$. With our least square estimates for a given $\bG$,  this implies the models defined by $\bG$ and $\bS\bG$ are equivalent. Let $\cS = Centr(\bF)$ the set (which is a $group$)  of all invertible matrices $\bS$ of size $\nmin\times \nmin$ such that $\bS\bF = \bF\bS$. We show that we can transform $\bG$ to normalized forms by applying an element of $\cS$, similar to the Gram-Schmidt process. In particular, we can make the rows of $\bG_{:, 0}$ orthonormal, similar to the classical Rayleigh quotient case. Since there are a few concepts to introduce, it is helpful to examine the case $p=2$ explicitly. In this case,  $\bG = \begin{bmatrix}\bG_{21} \\ \bG_{20} \\ \bG_{10}\end{bmatrix}$ and $\bF$ will have the form:
$$\bF = \begin{bmatrix}
0_{l_2} & \bI_{l_2} &  \\
  & 0_{l_2} & \\
 & & 0_{l_1}
 \end{bmatrix}$$
so for $\bS$ to commute with $\bF$, it needs to have the form:
$$
\begin{gathered}
\bS = \begin{bmatrix}
 \bS_{21, 21} & \bS_{21, 20} &  \bS_{21, 10} \\
& \bS_{20, 20} &  \\
   & \bS_{10, 20} & \bS_{10, 10} \\
 \end{bmatrix}\\
\bS_{21, 21} = \bS_{20, 20}
\end{gathered}
 $$
Each diagonal block is invertible, but there is no restriction on the off-diagonal blocks. Since the combined matrix $\begin{bmatrix} \bG_{2, 0} \\ \bG_{1, 0}\end{bmatrix}$ is of full row rank, we can make it orthonormal using an $LQ$ factorization (thin $QR$ on its transpose). The end results are matrices $\bS_{10, 10}$, $\bS_{10, 20}$ and $\bS_{20, 20}$, so that we have $(\bS\bG)_{1,0}(\bS\bG)_{1,0}' = \bI_{d_1}$, $(\bS\bG)_{2,0}(\bS\bG)_{2,0}' = \bI_{d_2}$ and $(\bS\bG)_{1, 0} (\bS\bG)_{2, 0}' = 0$. Therefore we assume after this step we have $\bG_{r, 0}\bG_{r_1, 0}' = \delta_{r, r_1}\bI_{r, r_1}$. We claim that we can choose the $\bS_{21, 20}, \bS_{21, 10}$ blocks of $\bS$ to make $(\bS\bG)_{2,1}(\bS\bG)_{10}' = 0$, $(\bS\bG)_{2,1}(\bS\bG)_{2,0}' = 0$. We have
$$(\bS\bG)_{2, 1} = \bS_{21, 10}\bG_{1, 0} + \bS_{21, 21}\bG_{2, 1} + \bS_{21, 20}\bG_{2, 0}$$
Note that $\bS_{21, 21} = \bS_{20, 20}$ is already defined in the first step. To make $(\bS\bG)_{2, 1}$ orthogonal to $\bG_{1, 0}$ and $\bG_{2, 0}$ we only need to set:
$$\begin{gathered}
\bS_{21, 10} = -\bS_{21, 21}\bG_{2, 1}\bG_{1, 0}' \\
\bS_{21, 20} = -\bS_{21, 21}\bG_{2, 1}\bG_{2, 0}'
\end{gathered}$$
This is our generalized $LQ$\slash Gram-Schmidt for $p=2$. While we cannot make $\bG$ fully orthogonal like the case $p=1$, this helps reduce the search space. Recall $h = \min(k, m)$ and $\fll = d_1 + d_2 \leq h$.
Assume we have chosen $\bG$ such that we have $\bG_{:, 0} \bG_{:, 0}' = \bI_{\fll}$ and $\bG_{:, 0} \bG_{2, 1}' = 0$. Completing $\bG_{:, 0}$ to an orthonormal basis by adding $m -\fll$ row vectors organized in a matrix $\bG_{\perp}$, then we can express $\bG_{2, 1}$ in that basis:
$$\bG_{2, 1} = \bC\Gperp$$
where $\bC$ is a matrix of size $(d_2, h-\fll)$. So far we have showed that for $p=2$ the choice of minimal state representation is the same as the choice of $d_1, d_2$ such that $d_1 + d_2 \leq m$, and $\bG$ could be normalized to an orthogonal form. This form could be represented by a pair $(\bC, \bO)$ where $\bO= \begin{bmatrix} \bG^o_{2, 0} \\ \bG^o_{1, 0}\\ \Gperp \end{bmatrix} $ is an orthonormal basis of $\R^m$ and $\bC\in \Mat(d_2, h-\fll)$.

We note if $\bQ_{20, 20}, \bQ_{10, 10}$ are orthogonal square matrices having $h-\fll, d_2, d_1$ rows respectively, then the block diagonal matrix $\bQ=\diag(\bQ_{20, 20}, \bQ_{20, 20}, \bQ_{10, 10})$ commutes with $\bF$. Hence
$(\bC, \begin{bmatrix}  \bG_{2, 0}\\ \bG_{1, 0}\\  \Gperp\end{bmatrix})$ and $(\bQ_{20, 20} \bC, \begin{bmatrix} \bQ_{20, 20}\bG_{2, 0}\\ \bQ_{10, 10}\bG_{1, 0}\\ \Gperp\end{bmatrix})$ represent $\bG$ and $\bQ\bG$, respectively. The generalized Rayleigh functional
$$\cR(\bG, \bA, \bB) = \frac{\det(\kappa(\bG)\bA\kappa(\bG)')}{\det(\kappa(\bG)\bB\kappa(\bG)')} $$
is invariant under multiplication of $\bG$ by $\bQ$. Also if we replace $\Gperp$ by $\bQ_{\perp}\Gperp$ and $\bC$ by $\bC\bQ_{\perp}'$ we get the same matrix $\bG^o$.

To illustrate, let us consider the case $k=m=2$ and $p=2$. Consider the case $d_2 = d_1 = 1$. In this case $\fll = 2$ and $\nmin= 3$ and $\Gperp$ is empty. An orthogonal matrix could be parameterized under the form $\begin{bmatrix}\cos t & \sin t\\ -\sin t & \cos t\end{bmatrix} = \begin{bmatrix} v \\ w\end{bmatrix}$. We can take $\bG = \begin{bmatrix}0 \\ v \\ w\end{bmatrix}$ and $\kappa(\bG) = \begin{bmatrix}v & 0 \\ 0& v \\0 & w\end{bmatrix}$. The numerator of the generalized Rayleigh quotient will be:
\[\det(\begin{bmatrix} vL^2 \bX \bA(L^2\bX)'v' & vL^2 \bX \bA(L\bX)'v' & vL^2 \bX \bA(L^1\bX)'w' \\
 vL \bX \bA(L^2\bX)'v' & vL \bX \bA(L\bX)'v' & vL \bX \bA(L\bX)'w' \\
  wL \bX \bA(L^2\bX)'v' & wL \bX \bA(L\bX)'v' & wL \bX \bA(L\bX)'w'
\end{bmatrix})\]
and the denominator is of the same form with $\bA$ is replaced by $\bB$. In effect we have an optimization problem on the circle, where the function to optimize is a rational function of high degree in $\sin$ and $\cos$.

The following proposition is rather technical but necessary for the exact parameter count of $\bS$ for the general case. The main point to remember is we can slide diagonally entries in a combined block to a wall. Within a combined block, the sub-blocks are equal if they are on the same (not necessarily principal) diagonal, and we only need to define $\bS$ on certain entries of vertical walls corresponding to $(\rho, 0)$ as in the case $p=2$.
\begin{proposition}
Let $\cS=Centr(\bF)$ be the set of all invertible $\nmin\times \nmin$ matrices commuting with $\bF$. We can index the blocks of $\bS\in \cS$ by $\bS_{\rho_1, j_1; \rho_2, j_2}$ for $1\leq \rho_i \leq p$ and $0 \leq j_i \leq \rho_i-1$. $\bS_{\rho_1, j_1; \rho_2, j_2}$  maps the block $\bG_{\rho_2, j_2}$ to block $(\bS\bG)_{\rho_1, j_1}$. $\bS_{\rho_1, j_1; \rho_2, j_2}$ is of size $d_{\rho_1}\times d_{\rho_2}$. We have the following characterization of $\bS$:
\begin{equation}
\label{eq:Sdiagonal}
\bS_{\rho_1, j_1; \rho_2, j_2} = \bS_{\rho_1, j_1+1; \rho_2, j_2+1} \text{ if } \rho_1 -1 > j_1 \geq 0, \rho_2 -1 > j_2 \geq 0
\end{equation}
\begin{equation}
\label{eq:Srelations}
\begin{gathered}
\bS_{\rho, 0; \rho, 0} \text{ is invertible.} \\
\bS_{\rho_1, j_1; \rho_2, \rho_2-1} = 0 \text{ if } \rho_1-1> j_1 \geq 0 \\
\bS_{\rho_1, 0; \rho_2, j_2} = 0 \text{ if } 0 < j_2 < \rho_2 \\
\end{gathered}
\end{equation}
As a consequence:
\begin{equation}
\label{eq:Szero}
\begin{gathered}
\bS_{\rho, \rho-1; \rho_2, j_2} = 0 \text{ if }j_2 > \rho-1 \\
\bS_{\rho_1, j_1; \rho, 0} = 0 \text{ if } \rho_1 -\rho > j_1\\
\bS_{\rho_1, 0; \rho_2, 0} = 0 \text{ if } \rho_1 > \rho_2
\end{gathered}
\end{equation}
The following blocks uniquely define $\bS$:
\begin{equation}
\label{eq:Sdefined}
\bS_{\rho_1, j; \rho, 0} \text{ with } j \geq \rho -\rho_1 
\end{equation}
Given a collection of blocks as in \cref{eq:Sdefined}, such that $\bS_{\rho, 0; \rho, 0}$ are invertible, we can construct a unique $\bS\in\cS$ using \cref{eq:Sdiagonal}. In particular, the number of parameters of $\bS$ is
\begin{equation}
\label{eq:Sdim}
\sum_{\rho_1,\rho}\min(\rho_1, \rho) d_{\rho_1}d_{\rho} = \sum_{i=1}^p(\sum_{j\geq i} d_j)^2
\end{equation}
To summarize, to define $\bS$ we only need to define the $(:, :, \rho, 0)$ vertical walls, and on those walls, for each $\rho_1$, $j$ could take values from $\max(0,\rho-\rho_1)$ to $\rho_1-1$. The remaining cells of $\bS$ are either zero, or could be filled by \cref{eq:Sdiagonal}. Finally:
\begin{equation}\kappa(\bS\bG) =\bS\kappa(\bG)
\end{equation}

\end{proposition}
\begin{proof}
Note that $\bF_{\rho_1, j_1; \rho_2, j_2}$ is zero, unless $\rho_1 = \rho_2$ and $j_1=j_2+1$. So $\bS\bF = \bF\bS$ implies:
$$(\bF\bS)_{\rho_1, j_1;\rho_2, j_2}=\begin{aligned}\bS_{\rho_1, j_1-1; \rho_2, j_2}\text{ if } j_1 > 0\\0 \text{ if } j_1 = 0\end{aligned}$$
$$(\bS\bF)_{\rho_1, j_1;\rho_2, j_2}= \begin{aligned}\bS_{\rho_1, j_1; \rho_2, j_2+1} \text{ if } j_2 < \rho_2-1\\ 0 \text{ if } j_2 = \rho_2 - 1\end{aligned}$$
From here \cref{eq:Sdiagonal} and \cref{eq:Srelations} follow.

For fixed $(\rho_1, \rho_2)$, consider the rectangular combined block:
$$\bS_{\rho_1, :; \rho_2, :} := (\bS_{\rho_1, j_1; \rho_2, j_2})_{\rho_1 > j_1 \geq 0; \rho_2 > j_2 \geq 0}$$
It has four walls corresponding to rows $\bS_{\rho_1, 0;\rho_2, :}$, $\bS_{\rho_1, \rho_1-1;\rho_2, :}$ and columns $\bS_{\rho_1, :;\rho_2, 0}$, $\bS_{\rho_1, :;\rho_2, \rho_2-1}$. By \cref{eq:Sdiagonal}, $\bS_{\rho_1, j_1; \rho_2, j_2}$ is defined if the surrounding walls are defined. From \cref{eq:Srelations}, the vertical wall $(:, :, \rho_2, \rho_2-1)$ is zero, except for $\bS_{\rho_1, \rho_1-1, \rho_2, \rho_2-1}$, and the horizontal wall $(\rho_1, 0, :, :)$ is zero, except for $\bS_{\rho_1, 0, \rho_2, 0}$. So $\bS$ is defined by the diagonal blocks and the $(\rho, \rho-1; :, :)$ horizontal walls as well as the $(:, :, \rho, 0)$ vertical walls. We have \cref{eq:Szero} by using \cref{eq:Sdiagonal} to move these matrix entries to another wall using \cref{eq:Sdiagonal}, then apply \cref{eq:Srelations}.

The first equation of \cref{eq:Szero} shows the only non-zero entries on the $(\rho, \rho-1)$ horizontal wall are those with $j_2 \leq \rho-1$, but then they are equal to $\bS_{\rho, \rho-1 -j_2; \rho_2, 0}$. So the entries of the vertical walls $(\rho, 0)$ alone are sufficient to define $\bS$. Finally, using the second equality of \cref{eq:Szero}, we have the restriction on $j$ in \cref{eq:Sdefined}.

We count the number of $j$'s for a given ordered pair of $\rho_1, \rho$ to be $\min(\rho, \rho_1)$, while the number of parameters for each such block is $d_{\rho}d_{\rho_1}$. From here the parameter count for $\bS$ follows.
The relationship between $\bS\kappa(\bG)$ and $\kappa(\bS\bG)$ could be verified by direct substitution. For $i \leq r -l$:
$$\kappa(\bS\bG)_{(r,l), i} = (\bS\bG)_{r, r-l-i}=\sum \bS_{r, r-l-i; r_3, l_3}\bG_{r_3, l_3}$$
Notice that row blocks $(r, l)$ of $\kappa$ are indexed in ascending order in $l$ while $\bS$ is ordered in descending order in $l$, so we need:
$$\sum_{r_2, l_2; i \leq r_2 - l_2} \bS_{r, r-l-1; r_2, r_2-l_2-1}\kappa(\bG)_{(r_2, l_2), i}= \sum_{r_2, l_2; i \leq r_2 - l_2} \bS_{r, r-l-1; r_2, r_2-l_2-1}\bG_{r_2, r_2-l_2-i}$$
With the change of variable $r_2 = r_3, l_3 = r_2-l_2-i$, the right-hand side becomes:
$$\sum_{r_3, l_3} \bS_{r, r-l-1; r_3, l_3+i-1}\bG_{r_3, l_3}$$
and then we apply $\bS_{r, r-l-1; r_3, l_3+i-1} = \bS_{r, r-l-i; r_3, l_3}$. We need to pay some attention to show the various constraints on indices carry through, but we will leave that to the reader.
\end{proof}

The generalized Gram-Schmidt (or $LQ$) algorithm is described next. It is purely a linear algebra result, we are not sure if it is already known.
\begin{proposition}
\label{prop:GGS}
For any block matrix $\bG\in \Mat(\nmin, m)$ such that $\bG_{:, 0}$ is of full row rank, we can construct an invertible matrix $\bS\in \cS=Centr(\bF)$ such that $\bG^{o} = \bS \bG$ satisfies
\begin{equation}
\label{eq:Orthonormal}
\bG^{o}_{:, 0} \bG^{o'}_{:, 0}  = \bI_{\fll}
\end{equation}
\begin{equation}
\label{eq:Orthogonal2}
\bG^{o}_{\rho, l} \bG^{o'}_{\rho_1, 0} = 0 \text{ for } l > \max(0, \rho-\rho_1-1)
\end{equation}
$(\bS_{\rho_1, 0; \rho_2, 0})_{\rho_1, \rho_2}$ could be chosen to be any block lower triangular matrix ($\bS_{\rho_1, 0; \rho_2, 0} = 0$ if $\rho_1 > \rho_2$) so that \cref{eq:Orthonormal} is satisfied, that is:
$$ (\bS_{\rho_1, 0; \rho_2, 0})_{\rho_1, \rho_2}(\bG_{\rho_2, 0})_{\rho_2} (\bG_{\rho_2, 0})_{\rho_2}' (\bS_{\rho_1, 0; \rho_2, 0})_{\rho_1, \rho_2}'  = \bI_{\fll}$$
Once $(\bS_{\rho_1, 0; \rho_2; 0})_{\rho_1, \rho_2}$ is chosen, $\bS$ is uniquely determined by \cref{eq:Orthogonal2}.
\end{proposition}
\begin{proof}
Applying the usual $QR$\slash Gram-Schmidt to $\bG'_{:, 0}$, then transpose, we get the $LQ$ factorization of $\bG_{:, 0} = \bL \bW_0$ with $\bW_0$ satisfies $\bW_0\bW_0' = \bI_{\fll}$ and $\bL$ is lower-triangular, so we can write $\bW_0 = \bL^{-1}\bG_{:, 0}$. We take 
$(\bS_{\rho_1, 0; \rho_2, 0})_{\rho_1, \rho_2}= \bL^{-1}$, which is lower-triangular, and $(\bG^o_{\rho_, 0})_{\rho }=\bW_0$.

From here, for each $\rho$ we get the diagonal blocks $\bS_{\rho, \rho-1; \rho; \rho-1} =\cdots = \bS_{\rho, 0; \rho, 0}$ such that $\bG^o_{\rho, 0} = \bS_{\rho, 0}\bG_{\rho, 0}$ satisfies: $\bG^{o}_{\rho, 0}\bG^{o'}_{\rho, 0} = I_{l_{\rho}}$ and $\bG^{o}_{\rho,0}\bG^{o'}_{\rho_1, 0}= 0$ for $\rho_1 \neq \rho$.
We note the block matrix:
$$(\bG_{\rho, 0}\bG^{o'}_{\rho_1, 0})_{{p\geq \rho\geq 1; p \geq \rho_1 \geq 1}}=\bL$$
is invertible. The remaining constraints on $\bS$, for fixed $\rho, j, \rho_1$, with $j> 0$ and $j\geq \rho - \rho_1$ are:
\begin{equation}
\label{eq:Ssolve}
\sum_{\rho_2, j_2} \bS_{\rho, j; \rho_2, j_2}\bG_{\rho_2, j_2}\bG^{o'}_{\rho_1, 0} = 0
\end{equation}
We need to show that we can solve for $\bS_{\rho, j; \rho_2, j_2}$ uniquely from the above equations. We will back solve in $j$. The case $j=0$ is already done, as by \cref{eq:Srelations} $\bS_{\rho, 0; \rho_2, j_2} = 0$ if $j_2 > 0$. With fix $j$, we try to solve \cref{eq:Ssolve} for each $\rho$'s blocks. If $j_2 >0$, by \cref{eq:Sdiagonal}, $\bS_{\rho, j; \rho_2, j_2} = \bS_{\rho, j-1; \rho_2, j_2-1}$ has been solved from the previous step. So we only need to solve for the case $j_2=0$, that is for $\bS_{\rho, j; \rho_2, 0}$. The constraint of $\bS$ forces $\bS_{\rho, j;\rho_2, 0} = 0$ unless $p\geq \rho_2\geq \rho-j$. Therefore we have only $p-\rho +j + 1$ block variables, corresponding to $p - \rho + j +1$ equations in \cref{eq:Ssolve}. The coefficient matrix in this case is $(\bG_{\rho_2, 0}\bG^{o'}_{\rho_1, 0})_{p\geq \rho_2 \geq \max(1, \rho-j); p \geq \rho_1 \geq \max(1, \rho-j)}$, which is invertible, and the invert is
$(\bS_{\rho, 0;\rho_1, 0})_{\rho, \rho_1}$ skipping the last $\max(1, \rho-j)-1$ rows and column blocks. So the solution exists and is unique. From \cref{eq:Sdefined}, $\bS$ is uniquely defined once we have $\bS_{\rho, j; \rho_2, 0}$.
\end{proof}
We implement this algorithm in function {\it LQ\_multi\_lag} in the package. With $\Psi$ and $\bG$ as inputs, it returns the factors $\bS$ and $\bW$ such that $\bG=\bS^{-1}\bW$, $\bS$ commutes with $\bF_{\Psi}$ and $\bW$ satisfies the orthogonal relations of \cref{eq:Orthonormal}, \cref{eq:Orthogonal2}. As an example for $\Psi=[(3, 1), (1,1)]$, we have $\bF=\bK(3,1)\oplus(\bK(1,1)$. With 
$$\bG = \begin{bmatrix}
-2 &  5 \\
 3 & -2 \\
 1 &  8 \\
1 & -2 \\
\end{bmatrix}
$$
the function found $\bG^o =\bS \bG$ with:
$$\bG^o = \begin{bmatrix}
 0. &  0.0 \\
-0.3969112 &  0.049614 \\
-0.1240347 & -0.992278 \\
-0.9922779 &  0.124035 \\
\end{bmatrix}; \bS = \begin{bmatrix}
-0.124035 & -0.024807 &  0.022326 & -0.195975 \\
0.000000 & -0.124035 & -0.024807 &  0.000000 \\
0.000000 &  0.000000 & -0.124035 &  0.000000 \\
0.000000 &  0.000000 & -0.186052 & -0.806226 \\
\end{bmatrix}
$$
The points to note are $\bG^o_{3, 0}$ and $\bG^o_{1, 0}$ are of norm $1$, $\bG^o_{3,2}=0$ and  $\bG^o_{3,1} = c \bG^o_{1, 0}$ ($c=0.4$ in this case). These are consequences of the orthogonal relations. So, while $\bG$ has total dimension $8$, $\bG^o$ just needs two parameters, one for the pair of orthogonal vectors, and one for the proportional constant between $\bG^o_{3,1}$ and $\bG^o_{1, 0}$.
The next proposition clarifies further the parameterization of $\bG^o$:
\begin{proposition}
\label{prop:Gparam}
Let $\cG$ be the set of matrices $\bG$'s of size $\nmin\times m$ such that $\bG_{:, 0}$ is of full row rank. Let $\cGO$ be the subset of $\cG$ consisting of all matrices $\bG^o$ satisfying the constraints in \cref{prop:GGS}. An element $\bQ$ in $\cS=Centr(\bF)$ maps $\cGO$ to $\cGO$ if and only if $\bQ$ is block diagonal with the diagonal blocks invertible and satisfies:
\begin{equation}
\label{eq:Q_format}
\begin{gathered}
\bQ_{r, 0; r, 0} = \bQ_{r, 1; r, 1} = \cdots = \bQ_{r, r-1; r, r-1} \\
\bQ_{r, 0; r, 0} \bQ_{r, 0; r, 0}^{\prime} = \bI_{d_r}
\end{gathered}
\end{equation}

Let us call $\cQ$ the set of all such $\bQ$'s. We have a pairing:
$\cQ \times \cGO\mapsto \cGO$ and the likelihood function is unchanged if $\bG^o$ is replaced by $\bQ\bG^o$.

Let $O(m)$ be the set of all square orthogonal matrices of size $m$. Set $d_0 = m -\fll$. For each matrix $\Gperp\in \Mat(d_0, m)$ such that $\GO \in O(m)$, ($\Gperp$ together with $\bG_{:, 0}$ form an orthonormal basis) we can find for $p \geq r \geq 1, r > l > 0$ matrices $\bC_{r, l}\in \Mat((d_r, \sum_{j=0}^{r-l-1}d_j)$ such that
\begin{equation}\label{eq:Go_rep}\bG^{o}_{r, l} = \bC_{r, l} \begin{bmatrix}\bG^o_{r-l-1}\\ \vdots\\ \bG^o_1\\ \Gperp\end{bmatrix} \end{equation}
Conversely, given $((\bC_{r, l})_{p\leq r \leq 1, l >0}, \bO)$, where $\bC(r, l) \in \Mat(d_r, \sum_{j=0}^{r-l-1}d_j)$ and $\bO\in O(m)$, we can reconstruct an element $\bG^o$ by first decompose $\bO$ to $\begin{bmatrix}\bO_{p, p-1}\\ \vdots\\ \bO_{1, 0} \\ \bO_{\perp} \end{bmatrix}$; then set $\bG^o_{r, 0}= \bO_{r, 0}$; $\bG^o_{r, l} = \bC_{r, l} \begin{bmatrix}\bO_{r-l-1}\\ \vdots\\ \bO_1\\ \bO_{\perp}\end{bmatrix}$ for $l > 0$. $\bG^o$ constructed that way is an element of $\cGO \subset \cG$.
\end{proposition}
Let us clarify that when $l = r-1$, \cref{eq:Go_rep} only has one block $\Gperp$.
\begin{proof} \cref{eq:Q_format} follows from \cref{eq:Sdiagonal} and the orthogonal relation of $\cGO$. In particular this means applying $\bQ$, the block $\bG_{r, 0}$ is transformed to $\bQ_{r, 0; r, 0}\bG_{r, 0}$ and thus orthogonal to $\bG_{\rho, 0}$ if $\rho \neq r$. The off diagonal blocks of $\bQ$ must satisfies an equation of the form \cref{eq:Ssolve}, and the orthogonal relation just mentioned means they are zeros. The fact that the likelihood function is unchanged by $\cQ$ follows from the fact that it is unchanged under $\cS$.

Given $\bG^o$ and $\Gperp$, we can take $\bC_{r, l}$ as the coefficients of $\bG^o_{r,l}$ in the basis $\begin{bmatrix}\bG^o_{r-l-1}\\ \vdots\\ \bG^o_1\\ \Gperp\end{bmatrix}$ by the orthogonal relations of \cref{prop:GGS}. On the other hand any linear combinations of that basis satisfies the orthogonal relation required by $\cGO$.
\end{proof}
For $j > 0$, $\bG_{:, j}$ could be arbitrarily large, $\cG$ and $\cGO$ are not bounded. However, we have the following proposition:
\begin{proposition}
\label{prop:bounded}
$\cR$ is bounded by maximum and minimum of $\frac{u\bA u'}{u\bB u'}$ with $u \in \Mat(\nmin, pm)$ is of full row rank with the later expression has maximum and minimum calculated by the generalized invariant subspace algorithm.
\end{proposition}
\begin{proof} This follows from the fact that the image of $\kappa$ is a subset of $\Mat(\nmin, pm)$.
\end{proof}
The example for $p=2$ above suggests that for the general case the likelihood function could be defined on a geometric object of higher dimension. We will return to this topic in \cref{sec:examples} as well as the appendix.

\section{Parameter reduction}
It is well-known that reduced-rank regression reduces the number of parameters by $(m-\fll)(k-\fll)$. We could see this by observing that $\bH$ and $\bG$ has $k\fll$ and $m\fll$ parameters, and we can replace $\{\bH, \bG\}$ by $\{\bH\bS^{-1}, \bS\bG\}$, where $\bS$ has $\fll^2$ parameters, so totally we have $(m+k)\fll - \fll^2$ free parameters, a reduction of $mk -(m+k)\fll + \fll^2$ free parameters versus full regression. For the state-space model we have:
\begin{proposition}
The total number of parameter reduction for state-space model of structure $\hat{\Psi}=[d_1,\cdots,d_p]$ is
\begin{equation}
\sum_{i=1}^p (k -\sum_{j\geq i}d_i)(m -\sum_{j\geq i}d_i)
\end{equation}
\end{proposition}
From this, we see there is no parameter reduction if $d_p = h=\min(p, k)$. The largest possible parameter reduction is $p(m-1)(k-1)$, which corresponds to $d_p=1, d_1=\cdots=d_{p-1}=0$. A change in $d_p$ has the most effect on the change in the number of free parameters.
\begin{proof}
We count the number of parameters of $\bH$ and $\bG$ and subtract by the number of parameter of $\bS$ to count the number of free parameter. The reduction is:
$$pmk - (m+k)(\sum j d_j) + \sum_{i=1}^p(\sum_{j\geq i} d_j)^2$$
which we see is the same expression as in the proposition.
\end{proof}
We plot the number of reductions versus minimal state-space dimension, and allocation rank $d_i$ in \cref{fig:param_reduction} for $m=10, p=5$. The averages are taken over all possible $\Psi$. We also fix four structural parameters and plot the fifth. The graph illustrates the point that we have the most parameter saving for higher exponent.
\begin{figure}
\includegraphics[scale=0.4]{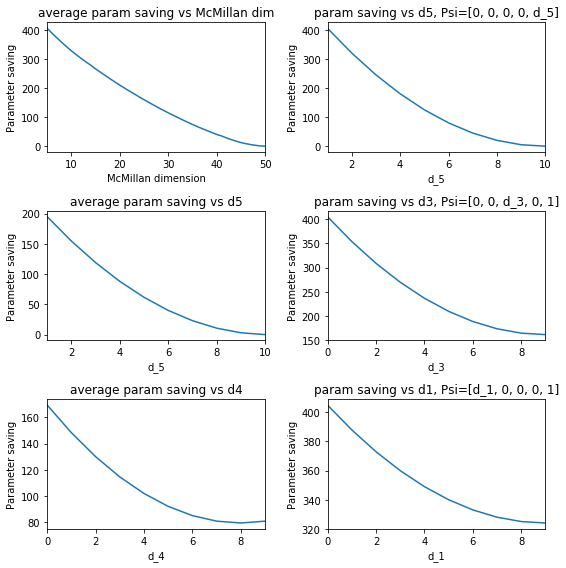}
\caption{Parameter reductions versus space-state dimension and structure parameters.}
\label{fig:param_reduction}
\end{figure}

\section{Examples.}
\label{sec:examples}
In \parencite{minmal_varx_code} we provided the python package implementing this model. The  open source notebook {\it minimal\_varx} in that package allows one to test the model on $colab$ environment without installing it on a local machine. We provided a number of examples in the notebook, which we would like to summarize the results here.

The main class in the package is $varx\_minimal\_estimator(\Psi, m)$. It can evaluate the likelihood function as well as its gradient given a matrix $\bG$. Given data matrices $\bX$ and $\bY$ we provide four data fitting method: {\it simple\_fit, gradient\_fit, hessian\_fit} and {\it manifold\_fit} (We implemented manifold\_fit for the case $\fll=h$ only, and use the orthogonal constraints but not the full reduction by symmetry of $\cQ$). The first three methods vectorize $\bG$ then just use a standard optimizer. All methods estimate $\bG$ to maximize the likelihood. The examples show that if $\Psi$ is known, the algorithm converges relatively fast. After fitting, $\bH, \bF, \bG, \bPhi$ could be read off the estimator. To forecast, we can use the {\it predict} method of the same class.

The package contains many utility functions, including those to produce the Smith-McMillan form for a polynomial matrix, as well as determinants of polynomial matrices and test for stability. We use a utility function to create a random stable polynomial of state-space form $\Psi$. We use this function to generate the tests samples. We provided a number of examples with different $m$ and $p$. The examples aim to clarify the concepts here. They also give a flavor of the behavior of this estimation method when the structure of $\bF$ changes. In our examples, we mostly work with $1000$ samples.

Let us first consider the case $m=2, p=2$. There are $\begin{pmatrix}p+m-1 \\ p\end{pmatrix} = 3$ possible configurations of $\Psi$. The case $\Psi= [(2, 2)]$, is the full rank case. There are two reduced-rank cases: $\Psi=[(2,1), (1, 1)]$ and $\Psi=[(2, 1)]$:

\subsection{The case $m=2, \Psi=[(2,1), (1, 1)]$: a circle}
In this case, both $\bG_{2, 0}$ and $\bG_{1, 0}$ has full rank $1$. In the first test, we randomly generate a number of stable matrices with structure parameter $\Psi$ then try to recover it using {\it simple\_fit}. We got reasonable convergence for our test. Applying \cref{prop:GGS}, $\bG_{2, 0}$ and $\bG_{1, 0}$ could be made to be orthogonal, so $\bG_{2, 0} = [\cos(t), \sin(t)]$ and $\bG_{1, 0} = [-\sin(t), \cos(t)]$. The generalized Rayleigh quotient is thus a ratio of two polynomial functions in $\cos(t)$ and $\sin(t)$, invariant when $\cos(t)$ is replaced by $-\cos(t)$, $\sin(t)$ by $-\sin(t)$ and hence it is sufficient to examine for $t\in [0,\pi]$. The enclosed graph plots the negative log likelihood for different values of $t$, and shows the minimum negative log likelihood is close to the likelihood of the data generation $\bG$, which is $-1.27233$.
\begin{figure}[h]
\includegraphics[scale=0.4]{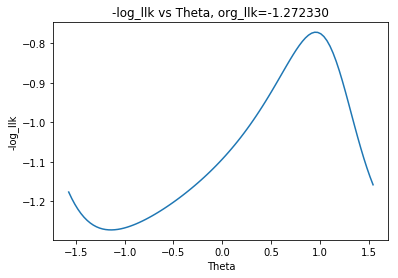}
\caption{$m=2, \Psi=[(2,1), (1, 1)]$, parameters versus minus log likelihood.}
\end{figure}
Below is the original versus the fitted $\bPhi$
$$
\bPhi_1^{\text{org}} = \begin{bmatrix}-0.21712203 & -0.5690077 \\
   0.43775564 &-0.98005326\end{bmatrix}
\bPhi_2^{\text{org}} =
\begin{bmatrix}
 0.12523273 &  0.04440429\\
  -0.25249089 & 0.17464151
\end{bmatrix}
$$
$$
\bPhi_1^{\text{fitted}} = 
\begin{bmatrix}
-0.19686572 & -0.53498222\\
   0.42740042 & -0.95142986
  \end{bmatrix}
\bPhi_2^{\text{fitted}} = 
 \begin{bmatrix} 0.12781235 & 0.05058104 \\
  -0.27748382 & 0.20364765\end{bmatrix}
$$
We also generate a random $\bG$ and show how $LQ$ factorization reduces it to orthogonal one.
\subsection{The case $m=2, \Psi=[(2,1)]$: a circle and its tangents.}
\label{subsec:circle_tangent}
In this case $\bG$ has the form $\begin{pmatrix}\bG_{2, 1}\\ \bG_{2, 0}\end{pmatrix}$. As before, in the first test we do not put the orthogonal restriction on $\bG$. We generate a number of stable polynomials with minimal state-space configuration $\Psi$ then recover them using {\it simple\_fit}. We get convergence as expected. The notebook also shows an example of $LQ$ factorization in this case.

This is a good example to illustrate the geometric concepts of the problem. Note that in this case $w = \bG_{2, 1}$ and $v= \bG_{2, 0}$ are both two-dimensional vectors. By \cref{prop:GGS}, $v$ could be assumed to have norm $1$, and $v.w = 0$. So we could think of $\bG_{2, 0}$ as being constrained to the unit circle, while for each $v$, $\bG_{2, 1}$ is being constrained to the line tangent to the unit circle at $v$. Therefore, the whole configurations of $\bG$ could be restricted to that of pairs $(v, w)$ of a point on the unit circle and a vector on its tangent line at $v$, not unlike the configuration space of position and velocity of a circular motion in classical physics problems. As before, if we parameterize $v = (\cos t, \sin t)$ then we can write $w= 
(-c\sin t, c\cos t) = cv_{\perp}$ with $v_{\perp} = (-\sin t, \cos t)$. So $O = \begin{bmatrix} -\sin t & \cos t \\ \cos t & \sin t  \end{bmatrix}$ as in \cref{prop:Gparam}. We plot the likelihood function on a two-dimensional surface of $(t,c)$, as well as plotting it for a number of fixed $t$ as below. The determinant ratio
In this example, the likelihood corresponding to $\bG$ used in data generation is $-1.101996$, vs estimated value $-1.096700$
\begin{figure}[h]
\includegraphics[scale=0.4]{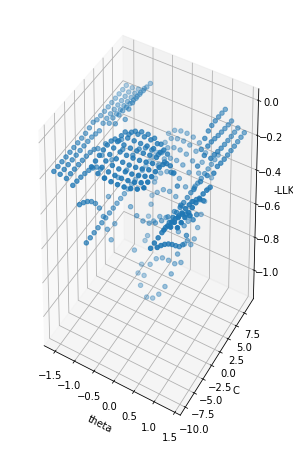}
\includegraphics[scale=0.4]{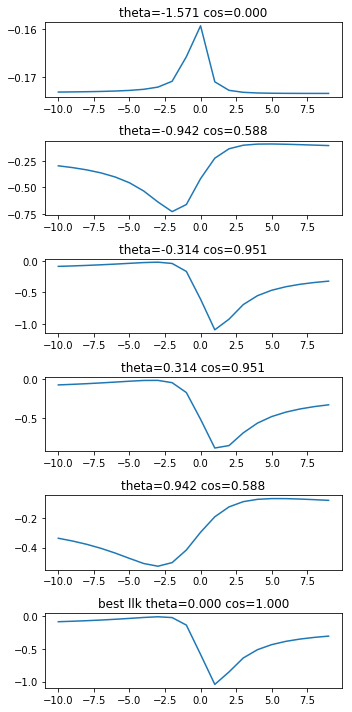}
\caption{Parameters versus minus log likelihood. Circle with tangent.}
\end{figure}
$$\bPhi_1^{\text{org}} = \begin{bmatrix}
 0.74186938 & -0.05524596 \\
  [-0.24536459 & 0.73052047
  \end{bmatrix}
\bPhi_2^{\text{org}} = \begin{bmatrix}
  0.12980835 & -0.05497244\\
  -0.04293259 & 0.14280693
  \end{bmatrix}
  $$
  $$
\bPhi_1^{\text{fitted}}
\begin{bmatrix}
 0.74142713 & -0.03649409\\
  -0.24909784 & 0.75546506
  \end{bmatrix}
  \bPhi_2^{\text{fitted}} = 
\begin{bmatrix}
  0.10750374 & -0.07386304 \\
  -0.03611811 & 0.13257722
  \end{bmatrix}
$$
\subsection{VAR(p) with $m=2$}
This is the generalization of the last two examples. As $AR(p)$ is well understood, $m=k=2$ is also the natural next step. $\Psi=[(p, 2)]$ is the case of full rank regression, which is already well studied, so we will assume $d_p = 1$. As before, there are two sub-cases, $\Psi=[(p, 1), (\rho, 1)]$ and $\Psi=[(p, 1)]$.  We show the reduction of the search space for $\bG^o$ by $LQ$ factorization discussed here in the notebook on our code page.

For the first case, the configuration space of $\bG^o$ is a circle with $p-\rho-1$ tangent vectors. We have $\bG^o_{p, 0}$ and $\bG^o_{\rho, 0}$ form an orthonormal basis. $\bG^o_{p, 1},\cdots,\bG^o_{p-\rho-1}$ are proportional to $\bG^o_{\rho, 0}$. In this case the full search space of $\bG$ is of dimension $2\nmin = 2(p+\rho)$ while the reduced space for $\bG^o$ has dimension $p-\rho$.
For the second case, $\bG^o_{p, 0}$ could be made of norm $1$, and  $\bG^o_{p, 1},\cdots ,\bG^o_{p, p-1}$  are orthogonal to it. So this case is a circle with $p-1$ tangent vectors. The minimal state-space dimension is $p$ and the search space for $\bG$ is of dimension $2p$, while the search space for $\bG^o$ is of dimension $p$.

This example illustrates the point that even when we can search on the full space of $\bG$ for smaller $m$ and $p$, the reduced space is of just half the dimension, but it is more complex to describe. We do not implement an optimization routine here but since we can parameterize the search space for $\bG^o$ explicitly, it could be done via the chain rule and a standard optimizer.

\subsection{$\Psi=[(p, d)]$ and the Velu-Reinsel-Wichern model.}
In \parencite{VeluReinselWichern}, the authors introduced a model of form
$$\by_t = A(L) B(L) L \bx_t +\bep_t$$
$\bx_t$ in their paper is $\bx_{t-1}$ in our notation. Here $A(L)$ is a $k\times d$ polynomial matrix function of degree $p_1$ and $B(L)$ is a $d\times m$ polynomial matrix function of degree $p_2$ (we switch $p_1$ and $p_2$ as used in their paper.) Set $p = p_1+p_2+1$. Consider the state-space model with $\bF = \bK(p, d)$. Set $\bG_{p, i}= B_{p_1-i}$ if $i \leq p_2$, and zero otherwise, $\bH_{p, i} = A_{p_2-i}$ if $i\leq p_1$, and zero otherwise. From \cref{eq:Phi}:
$$\bH(\bI-\bF L)^{-1}\bG = A(L) B(L) L$$
We note a number of blocks are set to zero in this model. The paper considers the case when $A(L)$ is constant, corresponding to the case where only $\bH_{p, 0}$ is non-zero. We can modify our framework to obtain the likelihood function for non constant $A(L)$. Note that the regressor for $A_i$ is $\sum B_{j} L^{j+i+1}\bX$ we get
$$[A_0,\cdots,A_{p_1}] = \bY\bX'_B(\bX_B\bX'_B)^{-1}$$
where $\bX_B = \upsilon(B)\Xlag$ and $\upsilon(B)$ is a block matrix of $(p_1+1)$ row $\times (p_1+p_2+1)$ column blocks.
\begin{equation}
\upsilon(B) = \begin{bmatrix} B_{p_2} & B_{p_2-1} &\cdots & B_{0} & 0 & \cdots & 0\\
0 & B_{p_2} &\cdots & B_{1} & B_{0} & \cdots & 0\\
\vdots & \vdots &\cdots & \vdots & \vdots & \cdots & \vdots\\
0 & 0 &\cdots & 0 & \vdots & B_1 & B_0\\
\end{bmatrix}
\end{equation}
We could think of $\upsilon$ as a truncated $\kappa$. For the case $p_1=0$ considered in their paper, $\upsilon(B)$ has only one row block and $p_2+1$ column blocks. With $\upsilon(B)$ in place of $\kappa(G)$, the result of \cref{th:likelihood} still applies, if $\upsilon(B)$ is of full row rank. We will assume this is the case, this means $\begin{bmatrix}B_{p_2} & B_{p_2-1} & \cdots & B_{0}\end{bmatrix}$ is of full row rank.

\subsection{Likelihood estimation over different $\Psi$'s.}
For $m$ and $p$ sufficiently large, the total number of configurations $\begin{bmatrix}h+p-1\\p\end{bmatrix}$ increases fairly quickly. The number of possible minimal state-space dimensions only increase linearly between $p$ and $hp$. A suggested strategy is not to iterate over all possible $\Psi$, but rather start with an $\nmin$ then search for $\bF$ (nilpotent but not necessarily Jordan) using continuous optimization method. However, it should be instructive to get a sense how $\Psi$, $\nmin$ and likelihood function interact. We plot the number of configurations of $\Psi$ per state-space dimension for $h=10, p=5$. The total number of distinct $\Psi$'s is 2002. The minimum states space dimension is between $5$ and $50$ and the number of $\Psi$ per minimum state-space dimension can be plotted to be of a bell-shape curve with the middle dimensions having the most number of $\Psi$, as in \cref{fig:PsiMcMillan}
\begin{figure}[h]
\includegraphics[scale=0.4]{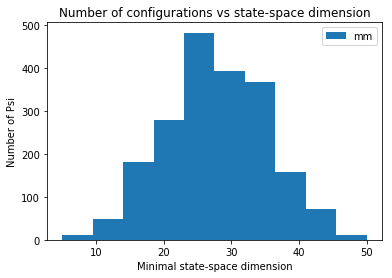}
\caption{Number of $\Psi$ v.s. McMillan degree}
\label{fig:PsiMcMillan}
\end{figure}

We also generate a stable matrix with $m=5, p=2, \Psi_{gen}=[(2, 2), (1, 2)]$. If we do not know $\Psi_{gen}$, we may need to search over the 15 possible $\Psi$ configurations, and we can summarize the optimization results in \cref{tab:all_psi}. From the table, once we get the correct $\Psi$, iteration over more complex $\Psi$ does not improve the likelihood. This suggests that we may not need to search over all $\Psi$, but aim to find a $\Psi$ with a sufficiently small state-space dimension with likelihood function sufficiently close to the full regression likelihood.
\begin{table}[H]
\begin{tabular}{lrrrrr}
\hline
{} &    2 &    1 &   org\_llk &  fitted\_llk &  success \\
\hline
0  &  1.0 &  0.0 & -8.746828 &   -5.575984 &      1.0 \\
1  &  1.0 &  1.0 & -8.746828 &   -6.715179 &      1.0 \\
2  &  1.0 &  2.0 & -8.746828 &   -7.614981 &      1.0 \\
3  &  1.0 &  3.0 & -8.746828 &   -8.290710 &      1.0 \\
4  &  1.0 &  4.0 & -8.746828 &   -8.342102 &      1.0 \\
5  &  2.0 &  0.0 & -8.746828 &   -7.745144 &      1.0 \\
6  &  2.0 &  1.0 & -8.746828 &   -8.569492 &      1.0 \\
7  &  2.0 &  2.0 & -8.746828 &   -8.754073 &      1.0 \\
8  &  2.0 &  3.0 & -8.746828 &   -8.754607 &      1.0 \\
9  &  3.0 &  0.0 & -8.746828 &   -8.754916 &      1.0 \\
10 &  3.0 &  1.0 & -8.746828 &   -8.756864 &      1.0 \\
11 &  3.0 &  2.0 & -8.746828 &   -8.759571 &      1.0 \\
12 &  4.0 &  0.0 & -8.746828 &   -8.759829 &      1.0 \\
13 &  4.0 &  1.0 & -8.746828 &   -8.760026 &      1.0 \\
14 &  5.0 &  0.0 & -8.746828 &   -8.760028 &      1.0 \\
\hline
\end{tabular}
\caption{Likelihood function for different $\Psi$.}
\label{tab:all_psi}
\end{table}
\subsection{Other examples}
\label{subsec:other_examples}
We ran an example for $m=7, k=5$ with $\Psi=[(2, 2), (1, 2)]$, again the fitted versus original likelihood are close ($19.87$ versus $-19.91$). In a final example, we ran $10$ tests with $k=8, p=3$, $\Psi=[(3, 1), (2, 1), (1, 2)]$. The fitted likelihood is smaller than the original likelihood as seen in \cref{tab:m8p3}. This seems to be an issue with autoregressive noise in the noise series used to generate the sample. We show also the full (i.e. no reduced rank assumption) regression likelihood, it fits with our minimal state-space likelihood well.
\begin{table}[H]
\begin{tabular}{lrrrr}
\hline
{} &    org\_llk &  fitted\_llk &   full\_llk &  success \\
\hline
0 &  -6.778974 &   -7.333462 &  -7.397742 &      1.0 \\
1 &  -7.664920 &   -8.280117 &  -8.329034 &      1.0 \\
2 &  -7.909143 &   -9.150681 &  -9.203443 &      1.0 \\
3 & -14.679623 &  -15.425846 & -15.491684 &      1.0 \\
4 &  -9.291283 &  -10.052678 & -10.098918 &      1.0 \\
5 & -12.626511 &  -13.004966 & -13.049997 &      1.0 \\
6 & -11.045051 &  -12.250215 & -12.296307 &      1.0 \\
7 &  -6.679380 &   -7.898204 &  -7.974920 &      1.0 \\
8 & -10.012286 &  -10.811005 & -11.115286 &      1.0 \\
9 & -64.521764 &  -65.580908 & -65.678105 &      0.0 \\
\hline
\end{tabular}
\caption{Case $m=8, p=3$}
\label{tab:m8p3}
\end{table}

\section{Discussion}
\subsection{An alternative space-space model}
We note that $L^{-1}\bT(L^{-1})^{-1}$ is also strictly proper, so we have a state-space realization:
$$L^{-1}\bT(L^{-1})^{-1} = \bH_a (L\bI -\bF_a)\bG_a$$
From here
$$\bT(L)^{-1} = \bH_a(\bI - \bF_a L)^{-1}\bG_a$$
For Vector Autoregressive model, we have a representation:
$$\bT(L)^{-1} = \bI - \sum\bPhi_i L^{i} = \bH_a (\bI - \bF_a L)^{-1}\bG_a$$
We note also zero is the only pole of $L^{-1}\bT(L^{-1})^{-1}$, so $\bF_a$ is again a Jordan matrix. Assuming we have the Smith-McMillan form:
$$L^{-1}\bT(L^{-1}) = A(L)S(L)B(L)$$
That means $A(L), B(L)$ are invertible polynomial matrices and $\bS(L)$ is diagonal satisfying the Smith-McMillan divisibility requirement. Then
$$L^{-1}\bT(L^{-1})^{-1} = B(L)^{-1}L^{-2}S(L)^{-1}A(L)^{-1}$$
$L^{-2}S(L)^{-1}$ is diagonal but not necessarily satisfying Smith-McMillan divisibility requirement, but we can make it to be, as in the final step of the Smith-McMillan algorithm. So $L^{-1}\bT(L^{-1})^{-1}$ and the traditional state-space form are intimately related. However, there is no direct link between our $\AR$-state-space realization and the traditional one. This alternative model is harder to estimate, even for $p=1$.

\subsection{Rank condition on $\bH$}
So far we recover $\bH$ by regression. Per Kalman, we should confirm the rank for $\bH_{:, 0}$. If it is not of full rank, the structure parameter $\Psi$ that we work with may not be minimal, and we can replace it by one with more reduced structure.

\subsection{Determining the structure parameters}
As $p$ and $m$ increases, the number of configurations for $\Psi$ increases, polynomially in $m$ and exponentially in $p$. Given that we have relatively fast convergence, a parallel search on configurations is certainly possible for a reasonable range of $p$ and $m$. However it may be unnecessary. The objective of the search should be for the configuration that balance between parameter reduction and close approximation to the full likelihood. As pointed out in earlier analysis, the number of parameter saving is impacted more by decreasing $d_i$ for a higher $i$. This motivates a search process where we do a full regression to obtain $\bPhi$, then applying a rank test to reduce the rank of $\bG_{:, 0}$ which penalizing higher exponents of the Jordan matrix. This could be done sequentially in descending order of exponent, stopping after a number of steps based on a balance between likelihood and parameter count. The search on each exponent could be done using a bisection search if $m$ is sufficiently large, and would have a $\log(m)$ iterative cost. So this method should be applicable even for large value of $m$ and $p$. This analysis could be done with the help of an information criteria or by a likelihood ratio criteria ($AIC$ or $BIC$). Another way is to formulate an objective function that could penalize a norm of $\bF^i$ for higher $i$. This may be a future research direction.

\subsection{Convergence Analysis}
It is well-known that Rayleigh quotient for a positive definite matrix is convex and has a unique minimum. By now, we know little about the analytic property of $\cR(\bG, \bA, \bB)$, except that its Hessian is known and it is bounded. For general $\bA$ and $\bB$ not necessarily constructed from the regression analysis here, it would be interesting to analyze the critical points of the $\cR$. One question would be if its minimal value is at a finite point, or would it be at a direction where $\bG(r, l)$ goes to infinite ($l > 0$). A second question would be if it has more than one local minima.

\subsection{Generalization to VARMA}
As Kalman's result addresses the multiple root case of minimal realization, a natural question is whether the results presented here has a full Gilbert-Kalman picture analogue. The answer is yes, which we will address in a forthcoming article. We hope the full result will give a new effective method in Linear System Identification.

\subsection{Further directions}
The approach could be adjusted to address the drift and seasonal adjustments. We have not addressed integration in this paper, however it seems plausible that it could be done with appropriate modification. A motivation for this paper comes from Johansen's approach to integration. Fixing a structure $\Psi$, we can study other loss functions depending on $\bG$ to go beyond the Gaussian assumption. Instead of $\bH$ applying linearly on $\kappa(\bG)L^i\bX$ we can assume a non-linear format. For example, we can use the kernel trick to replace $\Xlag\Xlag', \bY\bY'$ and $\bY\Xlag'$ with kernel values. We look forward to testing the model with real data.  We also look to improve on the optimization algorithms.

\begin{appendices}
\section{Vector bundle on flag manifolds}
To take full advantage of the invariant property of the likelihood function, manifold optimization may be an attractive option. In this appendix we summarize the results in term of flag manifolds. The uninterested reader can skip the appendix, consider it as a discussion on a particular optimization technique that help reduces the search space to a lower dimensional set taking advantage of the invariant property when replacing $\bG^o$ by $\bQ\bG^o$. On the other hand, the geometric picture could be thought of as a high dimensional generalization of the configurations of pairs of a particle moving on a circle and its velocity vector as explained in \cref{subsec:circle_tangent}.

Let us first fix a few notations.
\begin{itemize}
\item Recall $GL(m)$ is the group of all invertible matrices of size $m\times m$, $O(m)$ is the orthogonal group, $SO(m)$ is the special orthogonal group of all orthogonal matrices of size $m\times m$ with determinant $1$, $S(O(l_1)\times O(l_2)\times \cdots \times O(l_p) \times O(m -\fll))$ is the block diagonal subgroup of orthogonal group with block size $(l_1,\cdots, l_p, m -\fll)$ and determinant $1$.
\item Let $f_1 = l_1, f_i = \sum_{j=1}^{i} l_j$ and $f_{g+1} = m$. Consider $\cF(f_1, \cdots f_{g+1}; \R) = O(m)/(O(l_1)\times O(l_2)\times \cdots \times O(l_p) \times O(m -\fll))$. It is called a real flag manifold. It has an equivalent representation: $SO(m)/S(O(l_1)\times O(l_2)\times \cdots \times O(l_p) \times O(m -\fll))$, see \parencite{ye2019optimization}. (In the literature, it can also be considered as a quotient of $GL(m)$ by a parabolic subgroup of $GL(m)$).

\item Let $\cR(G, \bA, \bB) = \frac{\det(\kappa(G)\bA\kappa(G)'}{\det(\kappa(G)\bB\kappa(G)')}$ be the generalized Rayleigh quotient corresponding to two symmetric matrices $\bA, \bB$ each of size $mp \times mp$.
\end{itemize}
The following theorem describes the manifold that $\cR$ is defined on based on the invariant properties above, it may be just a restatement of results in \cref{sec:equivalence} in a fancier language, but it allows us to apply manifold optimization techniques:
\begin{theorem} The configuration space $\cGO$ and the group of orthogonal diagonal block matrices $\cQ$ have the following properties:
\begin{itemize}
\item $\cQ$ is isomorphic to $O(l_1)\times \cdots O(l_g)$. 
\item The map from $\cGO$ to $O(m)/ O(m-\fll)$ given by first representing $\bG^o$ as a pair $(\bC, \bO)$ with $\bC = (\bC_{r,l})_{r,l}$ and $\bO$ an orthogonal matrix as in \cref{prop:Gparam}, then map $\bO$ to the class $O(m-l)\bO\in O(m)/ O(m-\fll)$ is well-defined: two representations ($(\bC, \bO)$ and $(\bC_1, \bO_1)$ of $\bG^o$ give the same image in $O(m)/ O(m-\fll)$).
\item The above map induces a fiber bundle projection $\bpi$ from $\cGO/\cQ$ to $\cF(f_1, \cdots, f_{g+1}; \R)$. Each fiber is a vector space isomorphic to $\oplus_{r=1}^p \Mat(d_r, \sum_{l=1}^{r-1}\sum_{j=0}^{r-l-1}d_j)$ where $d_0=m-\fll$. So $\cGO/\cQ$ is a vector bundle over $\cF(f_1,\cdots, f_{g+1}; \R)$. We will call it $\cK(\Psi; m)$. When $p = 1$ or $\fll= m$, $\cK$ could be identified with $\cF$.
\item The dimension of $\cK$ is given by $m\sum_{j>0} j d_j - \sum_{i\geq 1}(\sum_{j\geq i}d_j)^2$
\item It is also given by $\sum_{i>j\geq 0}d_i d_j + \sum_{r=1}^p d_r \sum_{l=1}^{r-1}\sum_{j=0}^{r-l-1}d_j$.
\item $\cR$ is a bounded smooth function on $\cK(\Psi, m)$.
\end{itemize}
\end{theorem}
\begin{proof}
The first statement follows from the diagonal form of $\cQ$, and $\bQ_{\rho, l;\rho,l} = \bQ_{\rho}$ for $p-1\geq l \geq 1$.
The second statement is clear, as another representation of $\bG^o$ would have a form $(\bC\bQ_{\perp}^{\prime}, \begin{bmatrix}\bG_{:, 0} \\ \bQ_{\perp}\Gperp \end{bmatrix}$ for some $\bQ_{\perp}\in O(m-\fll)$. For the next statement, if $\bG^o$ is replaced by $\bQ\bG^o$ with $\bQ\in \cQ$, $(\bC, \GO)$ is transformed to 
$$(\bQ_r\bC_{r, l}\diag(\bQ'_{r-l-1}\cdots \bQ'_1,\bI_{d_0}))_{p\geq r\geq 1, r-1\geq l \geq 0}, \begin{bmatrix}\bQ_p \bG_{p, 0}\\ \vdots \\ \bQ_1 \bG_{1, 0}\\ \Gperp  \end{bmatrix}$$
So $\bpi(\bQ\bG^o)$ is in the same equivalent class with $\bpi(\bG^o)\in \cF(f_1,\cdots, f_{g+1}; \R)$. The fiber is isomorphic to $\bC=\oplus \bC_{r, l}$. The first expression for dimension of $\cK$ is just $\dim(\cG)-\dim(\cS)$, the second expression is sum of the dimensions of flag manifold and of the vector space fiber.

That $\cR$ is bounded is already proved in \cref{prop:bounded}, and it is clearly smooth.
\end{proof}
Optimizing on $\cK$ could be advantageous as in many instances its dimension is much less than $\dim(\cG) = m\sum jd_j$. We have seen in our examples the search space dimension could be reduced by half. If there were a manifold optimization package for $\cK$ we could take advantage of it. To our knowledge such a package is not yet available (however, see \parencite{ye2019optimization}. On the other hand, we can optimize on $O(m)\times \oplus_{r=1}^p \Mat(d_r, \sum_{l=1}^{r-1}\sum_{j=0}^{r-l-1}d_j)$, instead of taking quotient down to the flag manifold bundle level. We use the packages \parencite{manopt}, \parencite{JMLR:v17:16-177} to optimize on the $O(m)$  for the case where $\fll = h$ and offer a method {\it manifold\_fit} in our package.
\end{appendices}

\printbibliography[title={Bibilography}]
\end{document}